\documentclass[12pt,english]{article}
\usepackage[T1]{fontenc}
\include{step-bib,Impact-Zones-bib,myrefs}
\usepackage[latin9]{inputenc}
\usepackage{geometry}
\usepackage{amsthm}
\usepackage{amsmath}
\usepackage{amssymb}
\usepackage{graphicx}
\usepackage{setspace}
\usepackage{esint}
\theoremstyle{plain}
\newtheorem{thm}{\protect\theoremname}[section]
  \theoremstyle{plain}
  \newtheorem{lem}[thm]{\protect\lemmaname}
  \theoremstyle{remark}
  \newtheorem*{rem*}{\protect\remarkname}
  \theoremstyle{definition}
  \newtheorem{defn}[thm]{\protect\definitionname}
  \theoremstyle{plain}
  
  \theoremstyle{plain}
  \newtheorem{cor}[thm]{\protect\corollaryname}
  \theoremstyle{remark}
  
  \theoremstyle{definition}
  \newtheorem*{example*}{\protect\examplename}

\newcommand{\half}{\frac{1}{2}}

\usepackage{changepage}
\usepackage{caption}
\makeatother

\usepackage{babel}
  \providecommand{\claimname}{Claim}
  \providecommand{\corollaryname}{Corollary}
  \providecommand{\definitionname}{Definition}
  \providecommand{\examplename}{Example}
  \providecommand{\lemmaname}{Lemma}
  \providecommand{\propositionname}{Proposition}
  \providecommand{\remarkname}{Remark}
\providecommand{\theoremname}{Theorem}
\usepackage{color}

\begin{document}

\title{
% Springy block bouncing off a corner\\ or\\ \\
Impact Hamiltonian systems and polygonal billiards  }
\author{L. Becker, S. Elliott, B. Firester, S. Gonen Cohen, M. Pnueli, V. Rom-Kedar$^1$,\\ \normalsize  Department of Computer Science and Applied Mathematics, \\ \normalsize The Weizmann Institute of Science, Rehovot, Israel \\
 \normalsize $^1$The Estrin Family Chair of Computer Science and Applied Mathematics.}
\date{\today}
\maketitle

\LARGE
\textbf{Abstract}
\normalsize
The dynamics of a beam held on a horizontal frame by springs and bouncing off a step is described by a separable two degrees of freedom Hamiltonian system with impacts that respect, point wise, the separability symmetry. The energy in each degree of freedom is preserved, and the motion along each level set is conjugated, via action angle coordinates, to a geodesic flow on  a flat two-dimensional surface in the four dimensional phase space. Yet, for a range of energies, these surfaces are not the simple Liouville-Arnold tori - these are compact orientable  surfaces of genus two, thus the motion on them is  not conjugated to simple rotations. Namely,  even though energy is not transferred between the two degrees of freedom, the impact system is quasi-integrable and is not of the Liouville-Arnold type. In fact, for each level set in this range, the motion is conjugated to the well studied and highly non-trivial dynamics of directional motion in L-shaped billiards, where the billiard area and shape as well as the direction of motion vary continuously on iso-energetic level sets.   Return maps to Poincar\'e section of the flow are shown to be conjugated, on each level set, to interval exchange maps which are computed, up to quadratures, in the general nonlinear case and explicitly for the case of two linear oscillators bouncing off a step. It is established that for any such oscillator-step system  there exist step locations for which some of the level sets exhibit motion which is neither periodic nor ergodic. Changing the impact surface by introducing additional steps, staircases, strips and blocks from which the particle is reflected, leads to iso-energy surfaces that are foliated by families of  genus \(k \) level set surfaces, where the number and order of families of genus \(k\) depend on the energy.
\newpage

\section{Introduction}
\indent

Quasi-integrable dynamics appear in non-convex billiards with boundary consisting of horizontal and vertical segments  \cite{Zorich2002,Zorich2006,Athreya2012} and in non-convex billiards created by segments belonging to confocal quadrics  \cite{Dragovic2014,Dragovic2015,Fraczek2018}. The resulting dynamics are related to deep mathematical theories  on Interval Exchange Maps (IEM), on directed motion on translation surfaces, on genericity of curves in the space of affine lattices, on the Teichm\"uller geometry of moduli space  and even on some results in Number theory
%(e.g. regarding the gap distribution of fractional parts of the sequence %of square roots of positive integers numbers)
 (see \cite{DeMarco2011,Fraczek2018,Veech1982} and references therein).
We show that this fascinating collection of inter-related mathematical fields  are also related to the  rich research area of  Hamiltonian Impact Systems  (HIS). Thus, these theories are related  to a large variety of physically realizable models. We present this connection in the simplest possible setting and in the discussion we comment on some future synergetic directions.

In  \cite{Zorich2002,Dragovic2014} the  two known integrable billiards in the plane, rectangles and ellipses, are modified by considering  non-convex boundaries consisting of segments that respect the symmetries of the integrable billiard dynamics. The resulting tables, \textit{nibbled rectangles} \cite{Athreya2012}- domains defined by segments of horizontal and vertical boundaries  (the simplest non-trivial geometries are     slitted rectangle and L-shaped billiards, see, e.g. Fig \ref{fig:lshape})  and \textit{nibbled ellipses } \cite{Dragovic2014,Dragovic2015,Fraczek2018,frkaczek2019recurrence} - domains defined by segments of confocal quadrics,  display fascinating dynamical properties.  The nibbled rectangles are rational polygons and are thus analyzed by   constructing, by reflections along the horizontal and vertical segments, a flat surface (possibly with singularities). Then, the directional billiard flow on the nibbled rectangle is conjugated to the geodesic flow on the glued flat surface.  The genus of the flat surface is computable depending only on the number and type of corners (see \cite{Athreya2012} and references therein).  The return map to a transverse section of the surface is an IEM, and thus, the dynamics on the flat surface and the properties of the IEM are related. The dynamics on a given surface  depend on the direction of motion.
For a flat torus, the dynamics satisfy the Veech dichotomy: depending on the direction, either the motion is periodic or uniquely ergodic. The higher genus surfaces that are produced by the nibbled rectangles do not necessarily satisfy this condition; For the tables that do not produce lattice surfaces, there can be  directions of motion for which the dynamics are uniquely ergodic,  directions of motion such that a band of periodic trajectories co-exists with a band of trajectories that are dense on some set in the associated flat surface, and there can be also directions which are ergodic but not uniquely ergodic. Characterizing the measure of these sets of directions for a given billiard, the measure of parameters on which such behavior occurs for a given family of billiards, and defining proper statistical properties of the dynamics for such directions are delicate problems which are under current study, see e.g. \cite{DeMarco2011,Athreya2012,Sturman2012,Kim2017,Fraczek2018,frkaczek2019recurrence} and references therein.

In \cite{Dragovic2014,Dragovic2015,Fraczek2018,frkaczek2019recurrence} it was discovered that  the above tools may be applied to the study of the dynamics  in nibbled ellipses. Since reflections from confocal quadrics preserve the same integral of motion, for any fixed integral of motion a conjugacy to a directional motion on a glued flat surface is found, and, thus, an IEM can be constructed. Notably, each constant of motion (namely, each caustic) in a nibbled elliptic table defines a directional flow on a different flat surface while the direction of motion is fixed \cite{Fraczek2018}.  Recently it was established that under some conditions on the nibbled ellipse the family of directional flow on the resulting surfaces corresponds to a generic curve in the corresponding moduli space  \cite{Fraczek2018,frkaczek2019recurrence}.
We show here that  a rich class of HIS produces families of directional flow on flat surfaces, where both the direction and the geometry of the surfaces vary piecewise smoothly. While the question of conditions for genericity of the flow on iso-energetic surfaces remains open, the tools developed in \cite{Fraczek2018,frkaczek2019recurrence} appear relevant.

The field of Hamiltonian Impact Systems  (HIS), which corresponds to a smooth conservative motion in a domain \(D\) with elastic impacts from its boundaries, combines two types of dynamical systems - the non-trivial, possibly chaotic, smooth motion associated with Hamiltonian flows \cite{Arnold2007CelestialMechanics}, and, the dynamics resulting from elastic impacts, which have been extensively studied mainly in the context of billiards  \cite{kozlov1991billiards}. The combination of these two fields is  natural from a modeling point of view, as, in many systems, there is a smooth bounded interaction component (e.g. attraction between atoms) and short range repulsion (e.g. atomic repulsion between atoms) giving rise to steep potentials that may be approximated, as a singular perturbation, by elastic reflections  \cite{kozlov1991billiards,LRK12,KRK14,pnueli2018near}.  Analysis of non-integrable HIS includes local analysis near   periodic orbits of the HIS \cite{dullin1998linear,KRK14,Cheng1991}, analysis near homoclinic orbits of the HIS \cite{LRK12}, studies of the impact dynamics in some adiabatic limits  \cite{neishtadt2008jump,gorelyshev2006adiabatic},  persistence of KAM tori of motion along   convex boundaries \cite{zharnitsky2000invariant}, and even establishment of hyperbolic behavior for some specific type of systems of particles \cite{Wojtkowski1999,Woj98}.

 A class of HIS systems which is amenable to analysis under various perturbation is the Liouville Integrable Hamiltonian Impact Systems (LIHIS) - these are integrable Hamiltonian systems with impact surfaces which respect the integrability symmetries and for which the motion on almost all level sets is rotational:

\begin{defn}\label{def:liouvilimpint}
An HIS with compact level sets defined on a domain \(D\) belonging to a smooth manifold with piecewise smooth boundary is a Liouville-Integrable HIS (LIHIS) if:
\begin{description}
\item[RespF]  All the integrals of motion of the smooth Liouville-integrable Hamiltonian flow are preserved under impacts.

\item[Resp\(\theta\)]  The motion  on any connected component of a regular level set, namely, on components in the allowed region of motion on which the differentials of the constants of motion are independent, is conjugated to a directed motion on a torus.  \end{description}
\end{defn}
 For smooth systems (hereafter, by smooth we mean \(C^\infty\)-functions as in the Arnold-Liouville theorem, though, all results are  probably correct also for \(C^{r},r\geqslant2\) potentials) the Resp\(\theta\) condition follows from the RespF condition by the Arnold-Liouville theorem \cite{Arnold2013mathematicalmethods}, and in some works HIS satisfying the RespF conditions are called integrable \cite{Woj98}.  Examples for LIHIS  are mechanical separable\footnote{ Separable systems means hereafter  decoupled systems -  product systems of N 1 d.o.f mechanical Hamiltonians. The  more general class of separable systems defined in  \cite{marshall1988hamiltonian}  is not analyzed here.} Hamiltonians (i.e. \(H(q,p)=\sum_{i=1}^N(\frac{p_{i}^{2}}{2m}+V_{i}(q_i))\)) with compact level sets undergoing impacts from a perpendicular wall - an impact surface which is an infinite \(N-1\)-dimensional plane in the configuration space with a normal aligned along one of the \(q_{i}\) axes (see \cite{pnueli2018near,PRk2020} for the \(N=2\) case). Then, elastic impacts  with the wall send \(p_{i}\rightarrow-p_i\), and both the RespF and Resp\(\theta\) conditions are satisfied on regular level sets (as, in the \((q_{i},p_i)\) plane one can define action angle coordinates \cite{neishtadt2008jump} and all other d.o.f. are not affected, see  \cite{PRk2020}). As argued in  \cite{PRk2020}, it is expected  that separable systems with impact surfaces that consist of several such perpendicular walls are also LIHIS (e.g. a billiard in a rectangular box). This  class of LIHIS  enriches the number of integrable impact systems from merely two families for billiards (ellipsoidal billiards  \cite{Dragovic2011} and rectangular boxes) to the  huge class of all integrable separable Hamiltonian systems with perpendicular walls (and possibly to other integrable Hamiltonian systems with properly defined impact surfaces). Moreover in  \cite{pnueli2018near,PRk2020} it is establish that under some non-degeneracy assumptions on \(V_{i}\), perturbations of such 2 d.o.f. LIHIS by small smooth coupling terms and/or small smooth deformation of the walls are amenable to perturbation analysis (in particular to impact-KAM theory  \cite{pnueli2018near,PRk2020}).
The extension of these ideas to HIS with quadratic potentials in elliptic billiards \cite{Fedorov2001,Radnovic2015} is yet to be developed.

 Now  consider separable Hamiltonians  impacting from surfaces in the configuration space that are composed of several finite or semi-infinite planar plates, all of which are perpendicular to one of the \(q_{i}\)-axes. Then, elastic impacts  are of the same form, \(p_{i}\rightarrow-p_i\), so the RespF condition is still satisfied.  The HIS flow resides  on the intersection of the level set \(\{(q,p)|H_{i}(q_i,p_i)=e_i,  i=1,..,N\} \) with the billiard phase-space domain (i.e. with \((q,p)\in D\times\mathbb{R}^{n}\)), and the boundary created by the  impact surfaces is glued according to the elastic impact rule. For regular level sets of the smooth system these glued level sets are   \(N\) dimensional surfaces. Nonetheless, as shown next, in some cases the Resp\(\theta\) condition is not satisfied. We call such systems  Quasi-Integrable HIS (QIHIS), as we show that their dynamics is conjugated, on some of the level sets, to the directional motion on  Quasi-Integrable billiards.

To simplify the presentation we consider one of the simplest possible QIHIS:  two uncoupled oscillators that impact from a single right step in the configuration space, see Figure \ref{fig:models} for a physical realization of such a system; A   springy beam is held on a horizontal frame and reflects from a step. The springs are  connected to slider blocks that slide with no friction along a rectangular frame of ducts, and  the step-walls, marked by a dotted line are  out of the frame plane so that the slider blocks slide freely under the step walls and do not collide with them. The beam hits the step walls and bounces off them (always parallel to the out-of plane axis). The springs are rigid to bending (can extend only in one direction) and are uncoupled. Thus, as the beam bounces off the step-walls elastically, it retains its vertical and horizontal energies  \((e_{1},e_2)\).
Without the step, the system is a classical integrable system - all orbits belonging to a given level set \((e_{1},e_2)\) satisfy the Veech dichotomy: either all orbits are dense and cover the torus of angle variables uniformly (equi-distributed) or all orbits on this torus are closed. This behavior also implies that the return map to a transverse section of this torus is a rotation, and the rotation number on the prescribed level set determines which of the two options occurs. We show that this basic property is changed in the step-system. In particular, we identify a range of  iso-energy level set surfaces that are of genus two  and thus the return maps to a transverse section on such surfaces is, generally, a 5-IEM. This implies, for example, that an observable which depends on the oscillators phases (e.g. observable depending on the location of the beam) can have very delicate statistical properties \cite{Sturman2012,Kim2017}.

\begin{figure}
\begin{centering}
\includegraphics[scale=0.35]{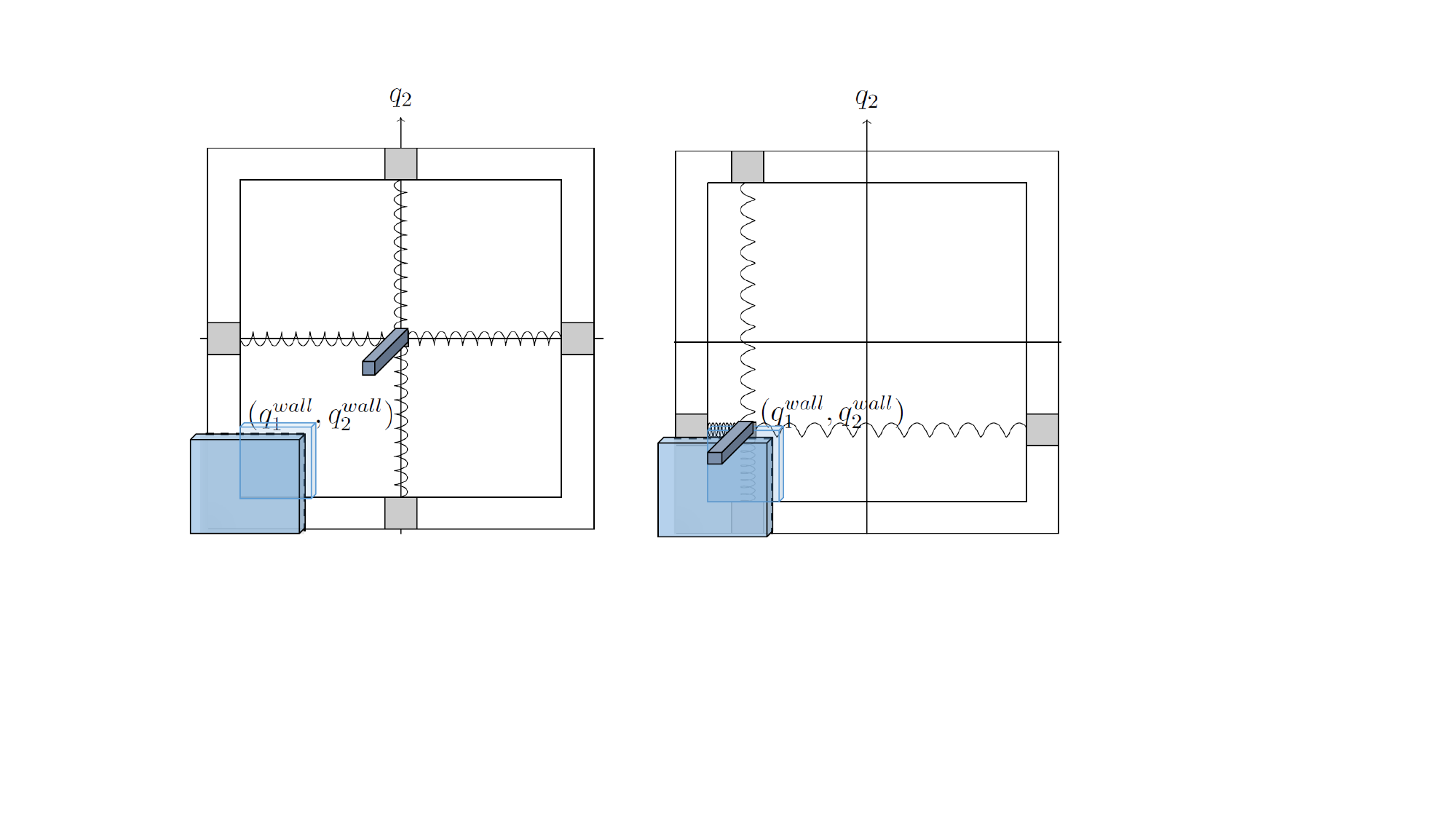}
\par\end{centering}
\protect\caption{A mechanical model for the Hamiltonian-step system  (Eqs. (\ref{eq:modelsham}) and (\ref{eq:stepdef})). (a) A rigid beam is confined between rigid horizontal and vertical springs hinged to supports that slide with no friction along a frame. (b) Two aligned rigid steps  are placed in front and on the back of  the frame, so when the beam hits these barriers an elastic impact occurs\label{fig:models}.}
\end{figure}

The paper is ordered as follows:
in section \ref{sec:main} we define the  step-system and state the main results: Theorem \ref{thm:mainflow} which conjugates the step dynamics to the quasi-integrable dynamics of directed motion on a compact orientable  surfaces of genus two and to  L-shaped billiards, Corollary  \ref{cor:nontrivialtop} which concludes that the energy surface has non-trivial foliation, Corollary \ref{cor:moresteps} which concludes that including additional steps, staircases, strips and rectangular scatterers can be similarly analyzed, Theorem \ref{thm:mainIEMsh} which concludes that  Poincar\`e return maps of the Hamiltonian impact step flow are conjugated, on each level set, to an IEM and Theorem \ref{thm:mainflowlin} which summarises the results for the case of linear oscillators with impacts from a step. In section \ref{sec:proofmain} we prove our main results - that the motion in this system, on each level set, is conjugated to either a billiard flow, in a specific direction, on a rectangular domain, or to a billiard flow, in a specific direction, on an L-shaped billiard.  Moreover, we prove that beyond a prescribed energy, the shape of the billiard on iso-energy surfaces changes from rectangular to continuously varying L-shaped billiards, back to rectangular domain, namely, that the topology of the invariant level set surfaces changes on iso-energy surfaces from genus one surfaces to genus two surfaces and back to genus one surfaces. In section \ref{sec:returnmaps}  we define and compute (up to quadratures) the corresponding Poincarè return maps for the level-set dynamics (Theorems  \ref{thm:mainmapR12} and \ref{thm:mainmap}), thus proving Theorem \ref{thm:mainIEMsh}. Section \ref{sec:moreproperties}  is devoted to establishing some specific properties of the resulting IEM, in particular, showing that typically there are many isolated level sets at which one of the intervals lengths vanishes. Section \ref{sec:linearosc}  applies these main results to linear oscillators, where the IEM are explicitly found, thus proving Theorem \ref{thm:mainflowlin}.   We end with a discussion in which we list several natural extensions of this work  and some open problems.

\section{\label{sec:main} The  step-system - setup and main results}
Consider an autonomous smooth  separable Hamiltonian corresponding to a particle motion  in \(\mathbb{R}^{2}\):
\begin{equation}
 H=H_1(q_1,p_1) + H_2(q_2,p_2)=\dfrac{p_1^2}{2m}+V_{1}(q_1)+\dfrac{p_2^2}{2m}+V_{2}(q_2) \label{eq:modelsham}\end{equation}where we assume for simplicity of presentation that the potentials \(V_{i}(q)\) have a single simple minimum  and are concave - they monotonically increasing to infinity as \(|q-q_{i,min}|\) increases (other interesting cases will be studies elsewhere). With no loss of generality,  take the particle mass to be \(m=1\) and the potentials minima to be at \(q_{i,min}=0  \) with \(V_{i}(0)=0\).   For positive i-energies \(e_{i}=H_i(q_i(t),p_i(t))>0, i=1,2\),  the particle oscillates with frequencies \((\omega_{1}(e_1),\omega_{2}(e_2))\) in the box    \([q_1^{min}(e_1),q_1^{max}(e_1)]\times[q_2^{min}(e_2), q_2^{max}(e_2)] \) of  the configuration space, where   \(q_i^{min}(e_i)\) and \(q_i^{max}(e_i)\) denote the minimal and maximal value of \(q_i(t)\) on the level set \(e_{i}\) (so \(V_{i}(q_i^{min}(e_i))=V_{i}(q_i^{max}(e_i))=e_{i}\)). Since  \(q_{i,min}=0  \), for all positive \(e_{i}\),  \(q_i^{min}(e_i)<0<q_i^{max}(e_i) \) and the level sets are nested \(\frac{d}{de_{i}}q_i^{max}(e_i)>0,\frac{d}{de_{i}}q_i^{min}(e_i)<0\)). Denote by \((\theta _{i}(t)=\omega_{i}(e_i)t+\theta _{i}(0),I_{i}(e_i))\) the action-angle coordinates of the 1 d.o.f. Hamiltonian  \(H_i(q_i,p_i)=H_i(I_i)\) (for one d.o.f. systems with concave potential these always exist and are unique up to a shift in the angle coordinate  \cite{Arnold2007CelestialMechanics}).
 A mechanical example of such a system is the beam held in a frame between two sets of  uncoupled springs hinged on  sliding, frictionless blocks (see Figure \ref{fig:models}). The simplest case to consider is of Linear Oscillators (LO), namely, the case of quadratic  potentials:\begin{equation}
V_{i}^{LO}(q_i)=\half \omega^{2}_{i}q_i^2, \quad i=1,2.\label{eq:LOpot}
\end{equation}We formulate below the main results (Theorems \ref{thm:mainflow}, Corollaries  \ref{cor:nontrivialtop} and \ref{cor:moresteps} and Theorem \ref{thm:mainIEMsh} ) for non-linear oscillators and dedicate Theorem \ref{thm:mainflowlin} and section \ref{sec:linearosc} to the LO case.

  Now, introduce a step  \(S\) in the configuration space (see Fig. \ref{fig:models}):\begin{equation}
S=\{(q_{1},q_{2})| \:q_1 < q_1^{wall} \text{ and }  q_2 <\label{eq:stepdef} q_2^{wall}\}, \qquad  q_1^{wall} \cdot q_2^{wall}\neq0,
\end{equation}
 and assume the particle bounces off elastically from this step (we require, to avoid degeneracies, that the step  is located away from the two axes); At the right wall of the step (hereafter, the 1-boundary), where \(q_{1}=q_1^{wall} \text{ and } q_2<q_2^{wall}\),
 the horizontal momentum is switched $
(q_1, q_2, p_1, p_2)\rightarrow\ (q_1, q_2, -p_1, p_2) $  whereas at the step upper wall (the 2-boundary), where \(q_{1}<q_1^{wall} \text{ and } q_2=q_2^{wall}\), the vertical  momentum changes sign  $
(q_1, q_2, p_1, p_2) \rightarrow\ (q_1, q_2, p_1,- p_2)$.
 When the particle hits the corner of the step the system is not defined and the trajectory stops. The flow is discontinuous at impacts,  is smooth elsewhere, and the vertical and horizontal energies, \(e_{i}\), are conserved by the impacts.
We
call this HIS the step system.
Denote the step energies by
\begin{equation}
h^{step}_{i}=V_i(q_i^{wall}), \qquad h^{step}=\ h^{step}_1+h^{step}_2,
\label{eq:histepdef}\end{equation}the step-family of level sets by:\begin{equation}
\mathcal{R}^{c}(h)=\{(e_{1},e_{2})|e_{1}\in(h^{step}_1,h-h^{step}_2),  \quad\ e_{2}=h-e_{1}\},\qquad\text{defined for } h>h^{step},
\label{eq:defsetRc}\end{equation}  \(\)by \(T_{i}(e_{i})=\frac{2\pi }{\omega_{i}(e_i)}\) the period of the smooth oscillators, by \begin{equation}\label{eq:Theta1smooth}
\Theta_{2}^{smooth}=\Theta_{2}^{smooth}(e_{1},h)=2\pi\frac{  T_{1}(e_1)}{T_{ 2}(h-e_{1})}
\end{equation} the rotation number of \(\theta_{2}\) on the level set \((e_{1},e_2=h-e_1)\), and by  \( \tilde T_{i}(e_{i};q_i^{wall})\) the period of the impact system when it is reflected from a wall at \(q_i=q_i^{wall}\) (namely,  \( \tilde T_{i}(e_{i};q_i^{wall})=
 2\int_{q_{i}^{wall}}^{q_i^{max}(e_i)} \frac{dq}{\sqrt{2(e_{i}-V_i(q_i))}} \)). Finally,  let\begin{equation}\label{eq:thetaiwalldef}
\theta_i^{wall}
(e_{i};q_{i}^{wall})=\pi\frac{\tilde T_{i}(e_{i};q_i^{wall})}{T_{i}(e_{i})}.
\end{equation}
  We will show later that by proper setting of the angle coordinate of the \(i\)th oscillator, \(\theta_i^{wall}
(e_{i};q_{i}^{wall})\) is the angle variable phase at the wall (see Lemma \ref{lem:reflectionsangles}).
 Depending on the properties of \(V_{i}\) and on the sign of \(q_{i}^{wall}\), the functions \(\theta_i^{wall}
(e_{i};q_{i}^{wall})\) may be monotone or not in \(e_{i}\) (for the LO case they are monotone, see below).
\begin{defn} \textit{The step system} is the two d.o.f.  HIS  defined by the smooth Hamiltonian of the mechanical form (\ref{eq:modelsham}) defined on \((q_{1},q_2)\in\mathbb{R}^{2}\setminus S\),\ with elastic reflections at the step \(S\) (Eq. (\ref{eq:stepdef})) boundaries. Each of the potentials \(V_{i}(q_i)\) in (\ref{eq:modelsham})  is smooth, has a single minimum at the origin and is concave:  \(qV_{i}'(q)>0\) for all \(q\neq0\).
\end{defn}

Our main results are  (see Fig \ref{fig:lshape}):
\begin{figure}
\begin{centering}
\includegraphics[scale=0.35]{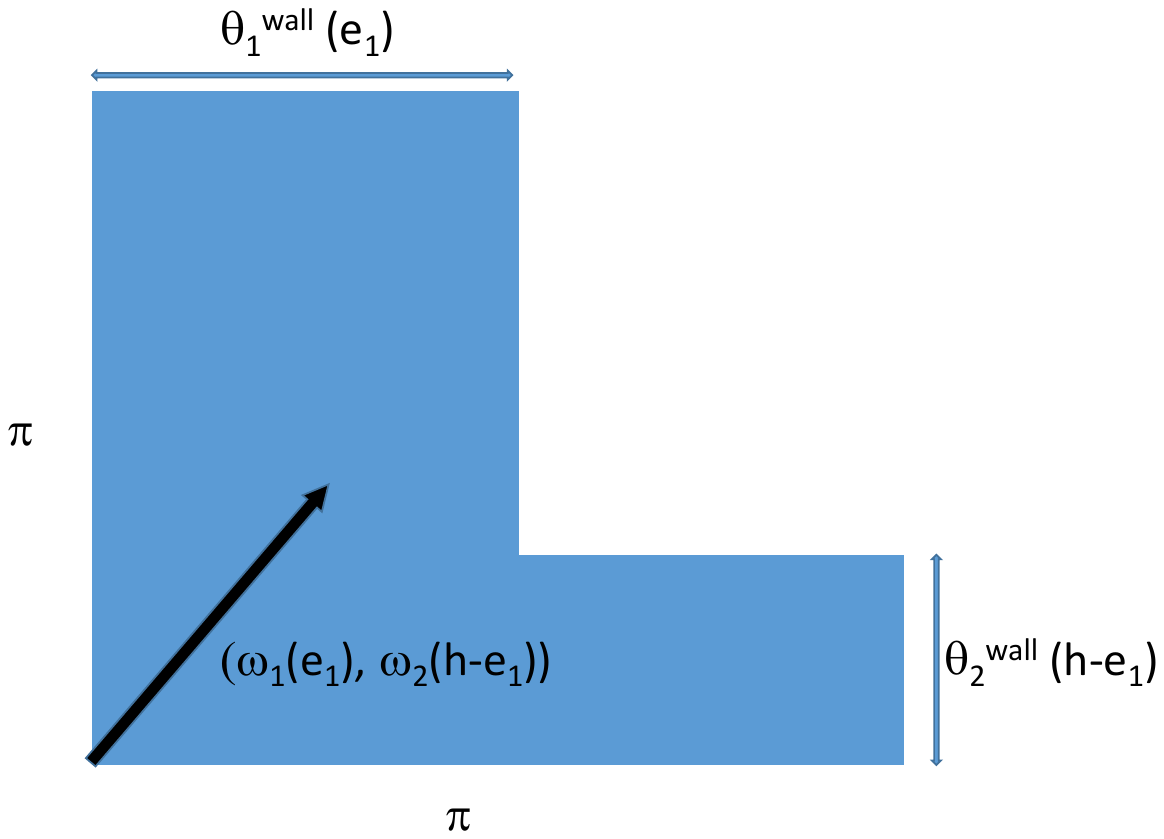}
\par\end{centering}
\protect\caption{The directional flow on the L shaped billiard \label{fig:lshape} \(L(\pi,\pi,\theta_{1}^{wall}(e_{1};q_{1}^{wall}),\theta_{2}^{wall}(h-e_{1};q_{2}^{wall}))\).}
\end{figure}

\begin{thm}\label{thm:mainflow}  The step system is  not Liouville-integrable HIS; For all \(h>h^{step}\),\ the flow on level sets belonging to the step family, \(  \mathcal{R}^{c}(h)\), is topologically conjugate to the  \((\omega_1(e_1),\omega_2(h-e_{1}))\)- directional billiard  flow on the   L shaped billiard  \(L(\pi,\pi,\theta_{1}^{wall}(e_{1};q_{1}^{wall}),\theta_{2}^{wall}(h-e_{1};q_{2}^{wall}))\) whereas the step flow for  the iso-energy level sets belonging to the complement of  \(\mathcal{R}^{c}(h)\)  is topologically conjugate to a directional billiard  flow on   a rectangular billiard. For all  \(h>h^{step}\) both families have positive measure.   \end{thm}

\begin{cor}\label{cor:nontrivialtop}For all \(h>h^{step}\) the foliation of the iso-energy surface
to level sets with increasing \(e_{1}\) value has two singularities at which the level sets topology changes: at   \(e_1= h_{1}^{step}\) the topology changes from a genus one surface to a genus two surface whereas at    \(e_1=h- h_{2}^{step}\) the topology changes back to a genus one surface. \end{cor}

\begin{cor}\label{cor:moresteps}By adding more steps, staircases, strips and rectangular barriers it is possible to create impact systems with level sets with any given genus \(\geq 1\) and any number of disconnected components. The corresponding iso-energy surfaces are foliated by a finite number of  families of  level sets with equi-genus and equi-number of components.  \end{cor}

\begin{thm}\label{thm:mainIEMsh} The return map of the step system to the section \(\Sigma_{1}=\{(q,p)|p_{1}=0,\dot p_1<0\}\)  for each  iso-energetic level set in \(\mathcal{R}^{c}(h)\) is conjugated to an interval exchange map of three intervals on a circle.  Restricting the angle to a natural fixed fundamental interval, for almost all level sets in  \(\mathcal{R}^{c}(h)\), the map becomes a five-interval exchange map (5-IEM). The return map to the section \(\Sigma_{1}\)  for  iso-energy level sets in the complement to \(\mathcal{R}^{c}(h)\) is a rotation on a circle, namely a 2-IEM.
\end{thm}

\begin{figure}
\begin{centering}
\includegraphics[scale=0.35]{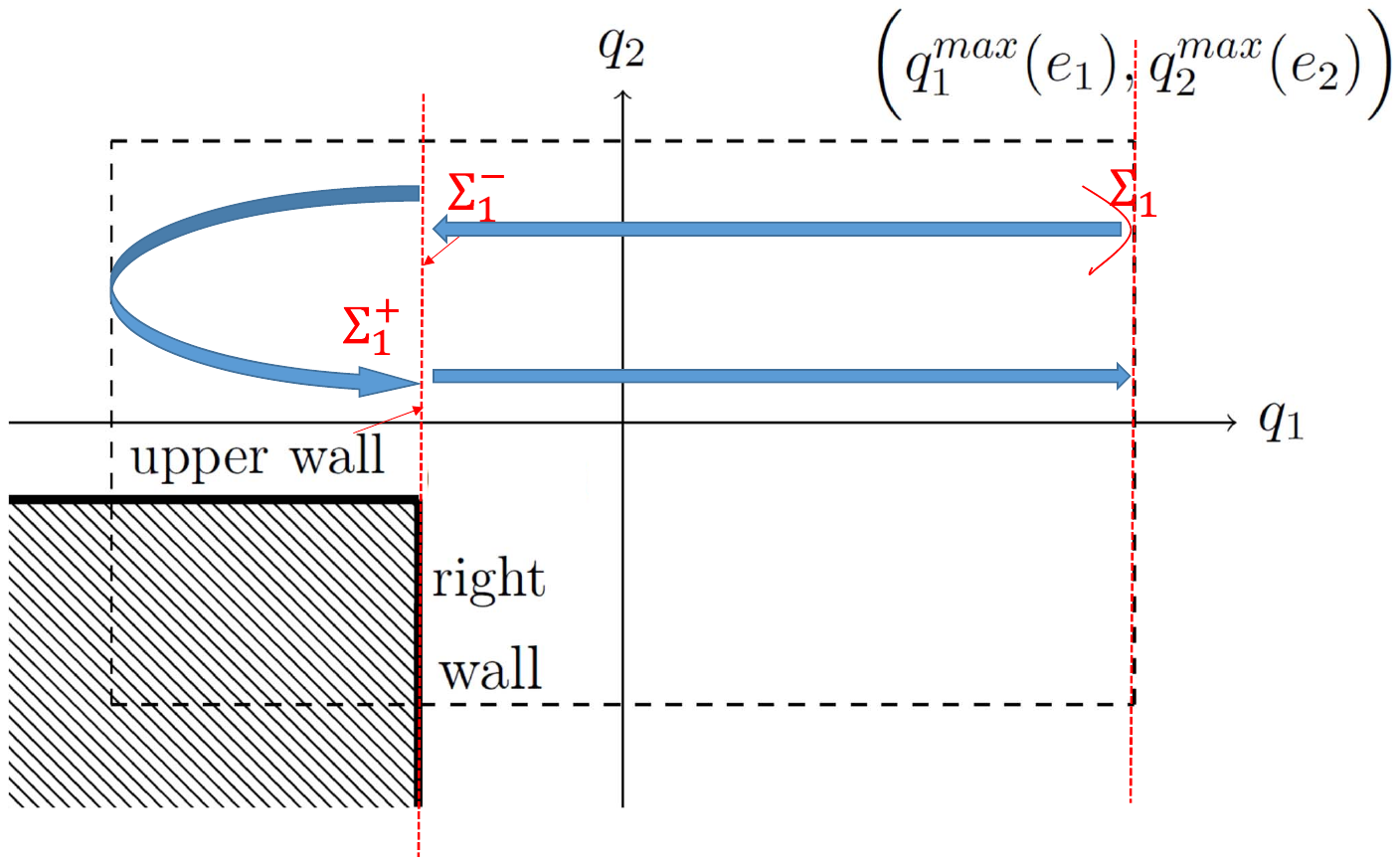}
\par\end{centering}
\protect\caption{\label{fig:returnmap}The return map geometry in configuration space.}
\end{figure}

In section \ref{sec:returnmaps},  explicit formulae (up to quadratures) for the return map at the iso-energy level sets \((e_{1},h-e_1)\) are derived (for concreteness we consider the return map to \(\Sigma_{1}\) -  the analogous computations for the return map to \(\Sigma_{2}\) amounts to replacing \(1\leftrightarrow2\) in all definitions, and the same conclusions apply). These computations show that the numerical properties of three functions of \(e_{1}\) (the functions \(\theta_{2}^{wall}(h-e_1),\Theta_{2}(e_{1},h),\chi_2(e_{1},h)\) defined by Eq. (\ref{eq:thetaiwalldef},\ref{eq:Theta},\ref{eq:psi})) determine the 5-IEM. In section \ref{sec:moreproperties} we discuss some properties of these functions and establish that there are isolated strongly resonant level sets at which orbits of different periods co-exist, level sets for which periodic and quasi-periodic motion co-exist, and,  isolated level sets in   \(\mathcal{R}^{c}(h)\) at which the IEM reduces to a rotation (at these values the level set surface is a lattice surface).
  We believe all the other level sets have minimal dynamics and almost all of them have uniquely ergodic dynamics. Proving this conjecture, namely the genericity of the iso-energy curve of directional L-shaped billiard flows as in \cite{Fraczek2018,frkaczek2019recurrence}, is beyond the scope of this paper.

For the linear oscillator step system, the period of the smooth motion does not depend on the energy, namely,  \(T_{i}(e_i)=\frac{2\pi }{w_{i}}\), and
 $q_i^{max}(e_i)=-q_i^{min}(e_i)=\sqrt{2e_{i}}/\omega_i$, hence all the functions which determine the dynamics are explicit:
\begin{align}
h^{step,LO}&=\half (\omega^{2}_{1}(q_1^{wall})^2+\omega^{2}_{2}(q_2^{wall})^2)\\
\theta_{i}^{wall,LO}(e_{i};q_{i}^{wall})&=\arccos \frac{\omega_{i}q_{i}^{wall}}{\sqrt{2e_{i}}}\in(0,\pi)\label{eq:thetailo}
\\
\Theta_{2}^{LO}&(e_{1})=2\frac{\omega_{2}}{\omega_{1}}\arccos \frac{\omega _{1}q_{1}^{wall}}{\sqrt{2e_{1}}}&\label{eq:Thetalo}
\\
\chi^{LO}_{2}&(e_{1},h)=\frac{\omega_{2}}{\omega_{1}}\frac{(\pi-\arccos \frac{\omega _{1}q_{1}^{wall}}{\sqrt{2e_{1}}})}{  \arccos \frac{\omega _{2}q_{2}^{wall}}{\sqrt{2(h-e_{1})}}} \label{eq:thetxiLO}
\end{align}

\begin{thm}\label{thm:mainflowlin} For \(h>h^{step,LO}\), the flow of the Linear-Oscillators-step system on each level set in  \(\mathcal{R}^{c}(h)\)   is topologically conjugated to the directional billiard flow in the fixed direction  \((\omega_1,\omega_2)\) on the L-shaped billiard \(L(\pi,\pi,\theta_{1}^{wall,LO}(e_{1};q_{1}^{wall}),\theta_{2}^{wall,LO}(h-e_{1};q_{2}^{wall})).\)    The L arms widths depend smoothly and monotonically on their arguments, and are of opposite monotonicity iff \(q_{1}^{wall}q_{2}^{wall}>0\).
The return map to the section \(\Sigma_{1}\)  is an IEM of the form (\ref{eq:explicitiemap}) with \((\theta_{2}^{wall},\Theta_{2},\chi_2)=(\theta_{2}^{wall,LO}(h-e_{1};q_{2}^{wall}),\Theta_{2}^{LO}(e_{1}),\chi_{2}^{LO}(e_{1},h))\) of  Eqs (\ref{eq:thetailo},\ref{eq:Thetalo},\ref{eq:thetxiLO}).   \end{thm}
The proof and other properties of the step LO are presented in section \ref{sec:linearosc}.

\section{The flow on level sets and the corresponding flat surfaces }\label{sec:proofmain}

\begin{figure}
\begin{centering}
\includegraphics[scale=0.3]{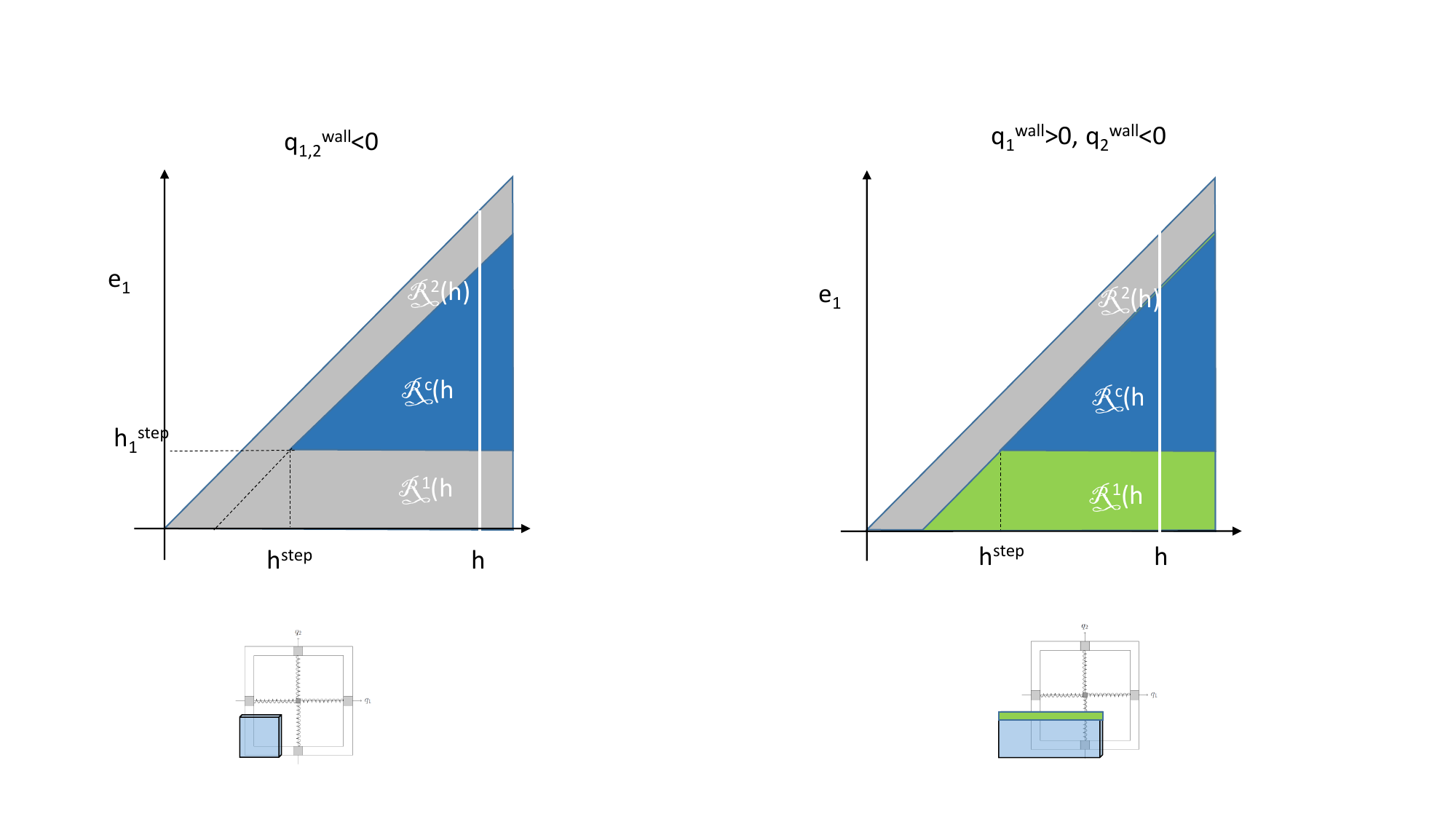}
\includegraphics[scale=0.3]{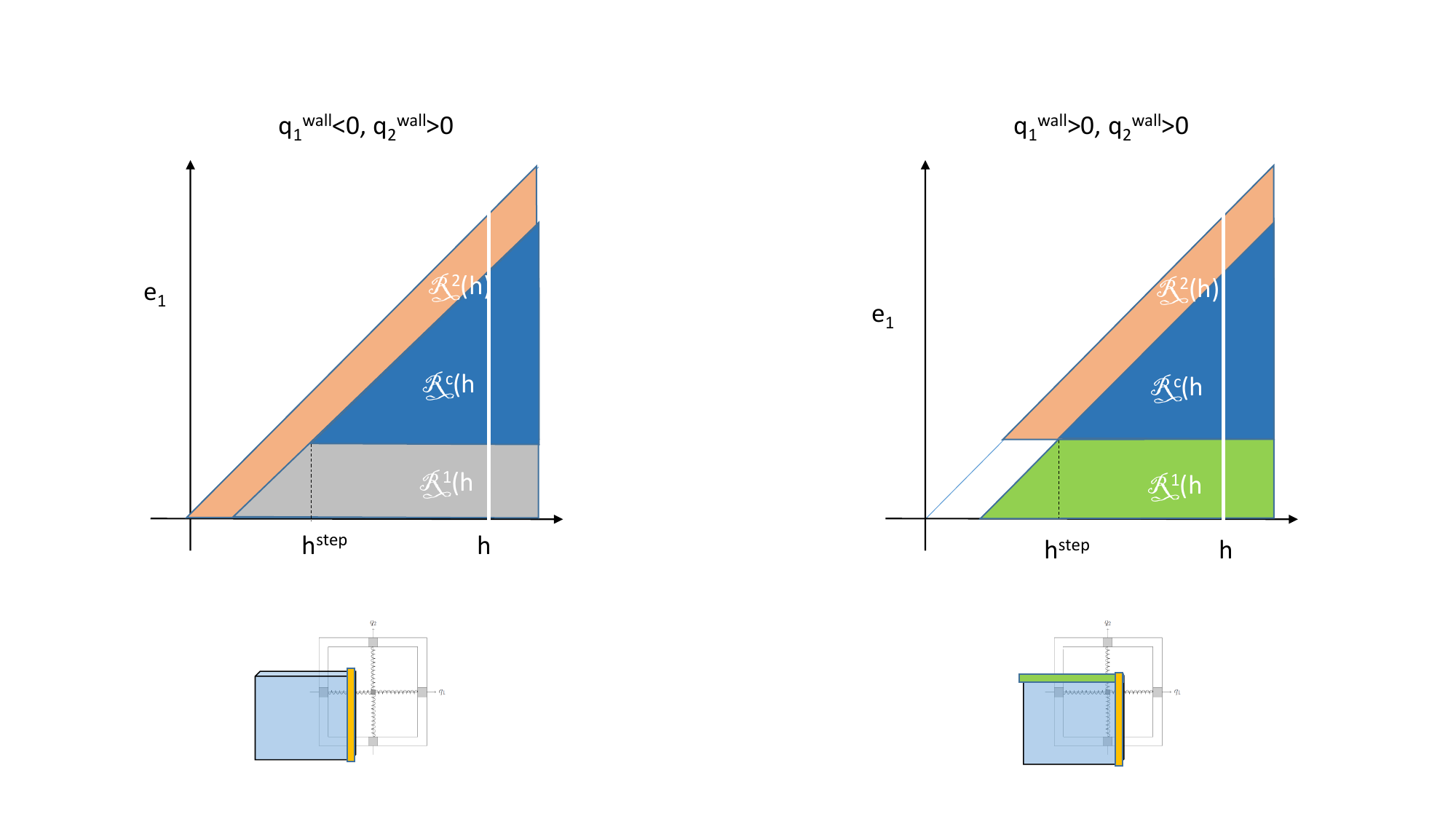}
\end{centering}
\protect\caption{\label{fig:embd}Impact-Energy momentum bifurcation diagram for the four relative positions of the step. The region where motion is allowed and no impacts occur (grey), the region where impacts occur at both sides of the step (blue) and the regions where impacts occur only at the upper (green) or right (orange) sides of the step are shown (see Lemmas  \ref{lemm:r12c}-\ref{lem:rcbothsides}).}
\end{figure}

In this section we prove Theorem  \ref{thm:mainflow}. The main observation is that in terms of the smooth action angle coordinates, for the proper range of energies (the region  \(\mathcal{R}^{c}(h)\)), impacts from the step correspond to a rectangular hole in the angle coordinates. Folding the torus according to the direction of motion in the configuration space leads to the motion in an L-shaped billiard with prescribed direction of motion and prescribed dimensions (up to quadratures).
The rotational motion in the complimentary regions to   \(\mathcal{R}^{c}(h)\) follows from realizing that in these regions, for each level set, either there are no impacts at all or all impacts occur with only one side of the step.

\noindent\textit{Proof of Theorem \ref{thm:mainflow}}: we first divide the level sets to three different classes according to the different types of impacts that may occur in each of  them (lemmas \ref{lemm:r12c}-\ref{lem:rcbothsides}). We then introduce the action-angle coordinates for the smooth system, fold them to the proper billiard table (an  L-shaped table for level sets in   \(\mathcal{R}^{c}(h)\) and a rectangular table  for the other level sets), and  establish that the impacts from the step in the flow are mapped to  impacts from the corresponding  boundaries of the billiard table.

\textbf{Delineating the energy level sets according to the impacts character:}

In the next few lemmas we detail how the collisions with the step depend on both the energy in each direction and on the location of the step. This classification, which is summarized by  Fig \ref{fig:embd} and its implications are shown in Fig. \ref{fig:levelsetboxes}, determines to which billiard table the flow on the level set is conjugated.
 Let \begin{equation}
\mathcal{R}(h)=\{(e_{1},e_{2})|e_{1,2} >0,\ e_{1}+e_{2}=h\},\qquad
\label{eq:defsetR}\end{equation} denote the open segment of allowed level set energy values on the isoenergy surface \(h\) (the white line in Fig \ref{fig:embd}) and by \(\mathcal{\bar R}(h)\) the corresponding closed interval. For all \(h>h^{step}\),  the isoenergy step-collision set,      \(\mathcal{R}^{c}(h)\), is an open segment  in the interior of    \(\mathcal{R}(h)\).  Define the two iso-energy complementary  closed segments:
\begin{equation}
\mathcal{\bar R}^i(h)=\{(e_{1},e_{2})|0\leqslant\ e_{i}\leqslant\  \min \{h,h^{step}_i\} , \quad\ e_{\bar i}=h-e_{i}\},
\label{eq:defsetRi}\end{equation}
(with interior open segments,  \(\mathcal{R}^{i}(h)\)), where, hereafter, we denote by \(\bar i\) the complement d.o.f. to \( i\) (namely \(\bar 1=2,\bar 2=1)  \).  Fig. \ref{fig:embd} shows these sets in the energy-momentum diagram for different locations of the walls.

\begin{figure}
\begin{centering}
\includegraphics[scale=0.2]{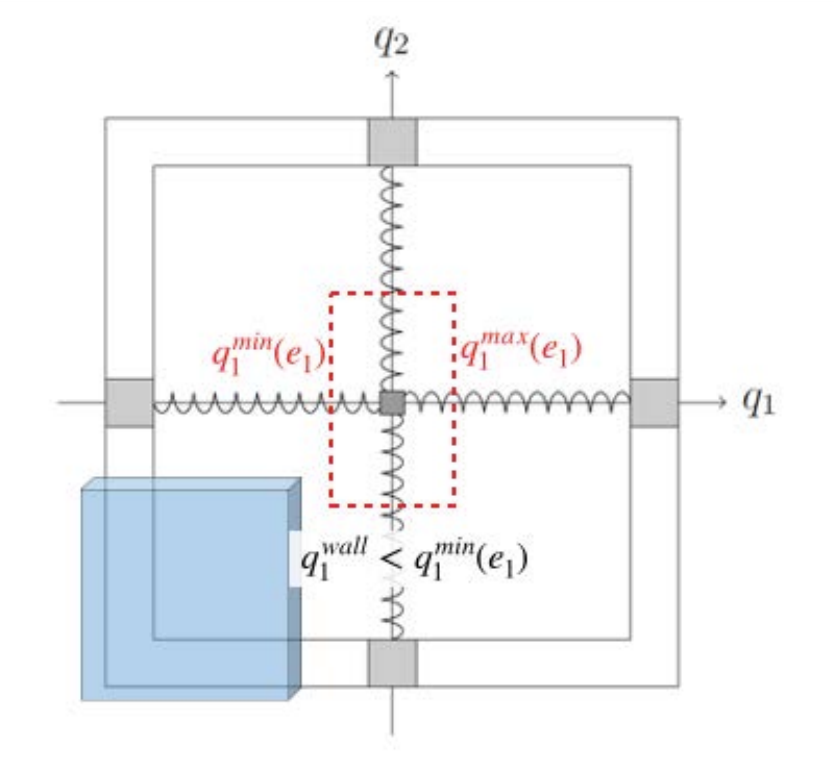}
\quad \includegraphics[scale=0.2]{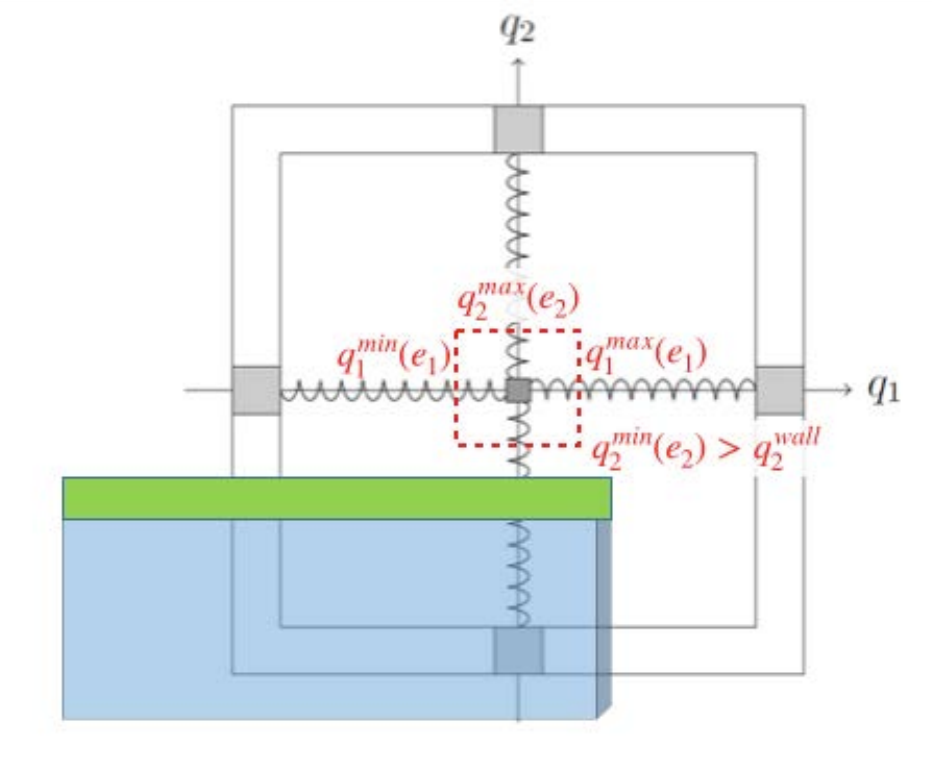}
\quad \includegraphics[scale=0.2]{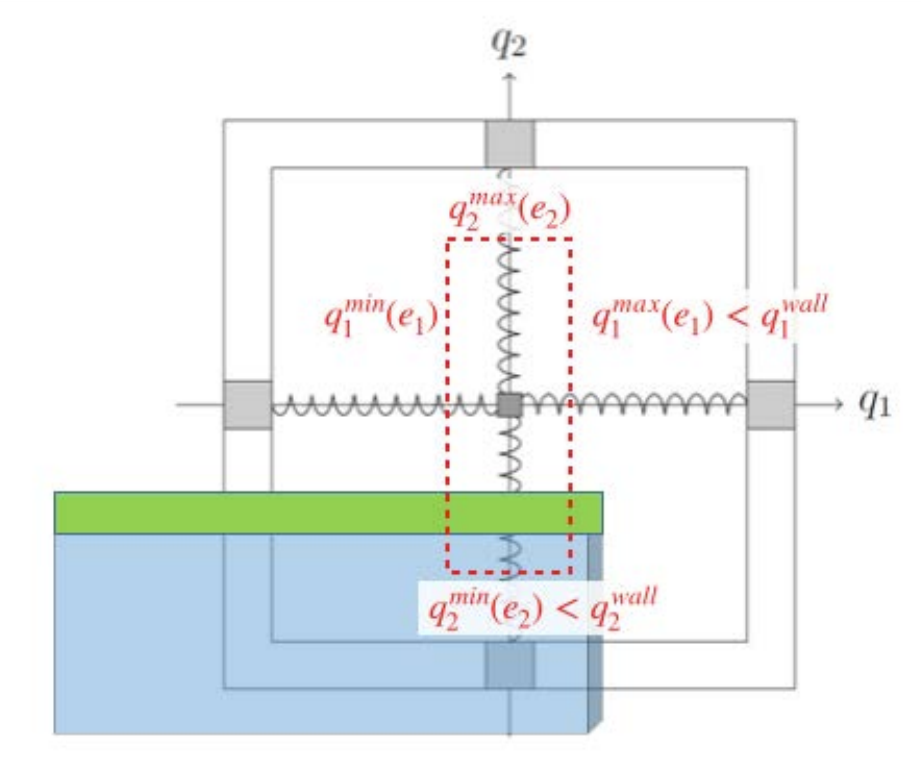}
\quad \includegraphics[scale=0.2]{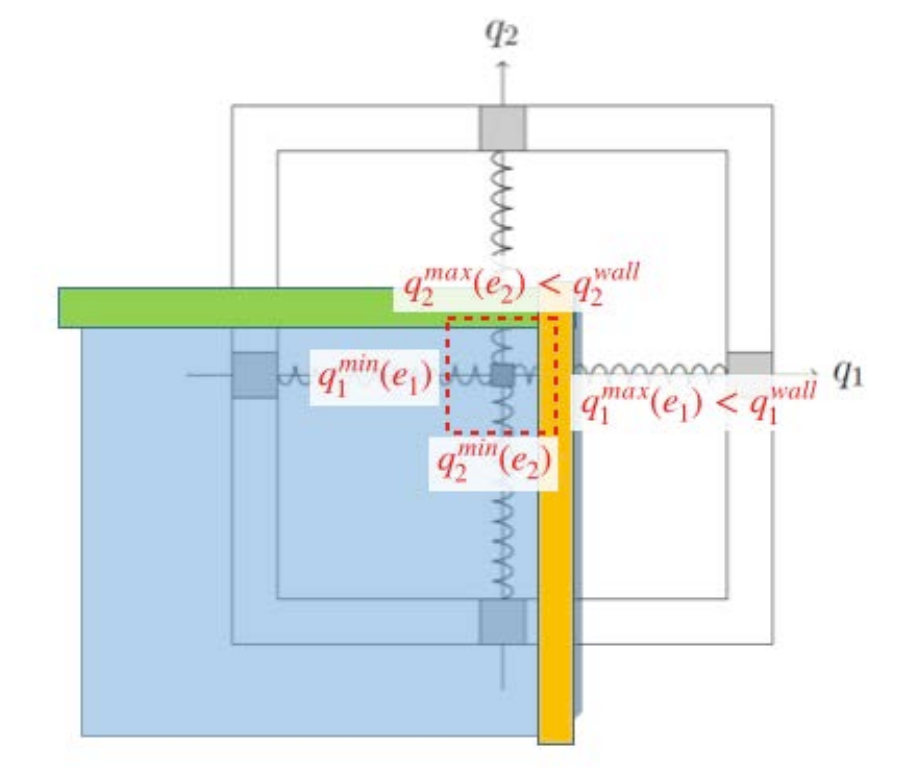}\\
\qquad \qquad \ \ (a) \qquad\qquad \qquad\quad\qquad (b) \quad \ \qquad \qquad\qquad\qquad (c) \ \ \quad \quad\qquad \qquad\qquad (d)
\end{centering}
\protect\caption{\label{fig:levelsetboxes} The level sets projection to the configuration space (dashed red box) for level sets in \( \mathcal{R}^1\) (see Lemmas  \ref{lemm:r12c} and \ref{lem:ridyn}); (a) No impacts:  \(q_{1,2}^{wall}<0, (e_{1},e_2)\in \mathcal{R}^1 \), (b) No impacts:  \(q_{1}^{wall}>0,q_{2}^{wall}<0, (e_{1},e_2)\in \mathcal{R}^1\cap \mathcal{R}^2 \),  (c) Impacts only with the 2-boundary (upper boundary)\ \(q_{1}^{wall}>0,q_{2}^{wall}<0, (e_{1},e_2)\in \mathcal{R}^1\setminus \mathcal{R}^2 \),  (d) No  motion for this level set: \(q_{1}^{wall}>0,q_{2}^{wall}>0, (e_{1},e_2)\in \mathcal{R}^1\cap \mathcal{R}^2 \).  }
\end{figure}
{\begin{lem}
\label{lemm:r12c}All trajectories belonging to level sets in   \(\mathcal{R}^{i}(h)\) do not hit the  \(i\)-boundary. For \(0<h<h^{step}\),   \(\mathcal{R}(h)=\mathcal{R}^1(h)\cup\mathcal{R}^2(h)\) and the segment  \(\mathcal{\bar R}^1(h)\cap\mathcal{\bar R}^2(h)\) is  non-empty. For all  \(h>h^{step}\),  \(\mathcal{\bar R}(h)=\mathcal{\bar R}^1(h)\cup\mathcal{R}^c(h)\cup\mathcal{\bar R}^2(h)\) and these three segments are non-empty and disjoint.
\end{lem}
 \begin{proof}Since the potentials are concave the level sets are nested. Level sets in the interior of  \(\mathcal{R}^{i}(h)\) satisfy \( e_{i}< h^{step}_i\), hence, for all \(t\), the trajectories satisfy:  \( q_i(t;e_i)\in[q_i^{min}(e_i),q_i^{max}(e_i)]\subset( q_i^{min}(h^{step}_i),q_i^{max}(h^{step}_i))\). By definition,    \(q_i^{wall}\in\{q_i^{min}(h^{step}_i),q_i^{max}(h^{step}_i)\} \) so such trajectories do not cross the line \(q_{i}=q_i^{wall}\) and the step \(i\)th boundary cannot be impacted.  The rest of the lemma  follows from the definitions of  \(\mathcal{R}(h),\mathcal{R}^{1,2}(h),\mathcal{R}^c(h)\) (Eqs.  (\ref{eq:defsetRc}),(\ref{eq:defsetR}),(\ref{eq:defsetRi})), see also Fig. \ref{fig:embd}.
\end{proof}
   Fig. \ref{fig:levelsetboxes} demonstrates   that  in accordance with lemma \ref{lemm:r12c}, level sets that belong to \(\mathcal{R}^1(h)\) do not impact  the \(1\)-boundary (the right side of the step).
Next we establish when such level sets impact the  \(2\)-boundary (the upper side of the step).:
\begin{lem}\label{lem:ridyn}
{  If  \(q_{i}^{wall}<0 \), trajectories associated with level sets in    \(\mathcal{R}^{i}(h)\)  do not hit the step.
If   \(q_{i}^{wall}>0,  \) the dynamics in   \(\mathcal{R}^{i}(h)\)  is further divided to the following two cases: \begin{description}
\item[ For level sets in   \(\mathcal{R}^{i}(h)\backslash\mathcal{\bar R}^{\bar i}(h)\):] trajectories    hit the  \(\bar i\)-boundary only, and the impacts are transverse.
\item[ For level sets in    \(\mathcal{ R}^1(h)\cap\mathcal{ R}^2(h)\):]
  trajectories  do not hit the step if   \(q_{\bar i}^{wall}<0 \) and are not in the allowed region of motion if    \(q_{\bar i}^{wall}>0 \).
\end{description}}
      \end{lem}

\begin{proof}If  \((e_{1},e_{2})\) belong to   \(\mathcal{R}^i(h)\) then \(e_{i}<h_i^{step}\) (see (\ref{eq:defsetRi})). If additionally,  \( q_{i}^{wall}<0\), then   \(q_i^{min}(e_{i})> q_{i}^{wall}\),  so the oscillation in the \(i\)th direction do not reach the wall, independently of the oscillation amplitude in the \(\bar i\) direction (see Fig \ref{fig:levelsetboxes}(a)).

If \(q_{i}^{wall}> 0\),  then \(q_i^{max}(e_{i})< q_{i}^{wall}\), so, while impacts cannot occur with the \(i\) boundary, transverse impacts with the    \(\bar i\) boundary occur  when\( \ e_{\bar i}>h_{\bar i}^{step}\), namely when \((e_{1},e_{2})\in\mathcal{R}^{i}(h)\backslash\mathcal{\bar R}^{\bar i}(h)\) (see Fig \ref{fig:levelsetboxes}(b)).

If  \(q_{i}^{wall}> 0\) and \( \ e_{\bar i}<h_{\bar i}^{step}\), so    \((e_{1},e_{2})\in\mathcal{R}^1(h)\cap\mathcal{R}^2(h),\)  the \(\bar i\) boundary cannot be crossed. If, additionally,  \(q_{\bar i}^{wall}< 0\), then \( \ e_{\bar i}<h_{\bar i}^{step}\) implies that \(q_{\bar i}^{min}(e_{\bar i})> q_{\bar i}^{wall}\) and the oscillations are in the allowed region of motion and do not hit the step (see Fig \ref{fig:levelsetboxes}(c)), whereas if   \(q_{\bar i}^{wall}> 0\) then \(q_{\bar i}^{max}(e_{\bar i})< q_{\bar i}^{wall}\) and the motion  is "behind the step" namely it is not in the allowed region of motion (see Fig \ref{fig:levelsetboxes}(d)). \end{proof}

\begin{lem}\label{lem:rcbothsides}
{  Each level set in the step collision set,  \(\mathcal{R}^{c}(h)\), includes trajectories which  impact transversely the 1-boundary and trajectories which impact transversely the 2-boundary.
  }   \end{lem}
\begin{proof} Consider \((e_{1},e_{2})\in\mathcal{R}^{c}(h)\). Then,    the corresponding level sets in each d.o.f. include the step position, namely, \(q_i^{min}(e_{i})< q_{i}^{wall}<q_i^{max}(e_{i}),\ \ i=1,2\).  Denote hereafter the smooth Hamiltonian  flow by \(\varphi_{t}^{smooth}(z)\) where \(z=(q_1,q_2,p_1,p_2)\). The open, one dimensional set of i.c. \( Z_{1}=\{z|z=\left(q_1^{wall},q_2,-\sqrt{2(e_{1}-h_{1}^{step})},\pm\sqrt{2(e_{2}-V_{2}(q_2))}\right),q_2\in (q_2^{min}(e_{2}),q_{2}^{wall})\}\)
  is non-empty and belongs, by construction, to the level set \((e_{1},e_2)\). Its projection  to the configuration space belongs to the right, 1-boundary of the step. Hence, for sufficiently small \(t\), the set \(\varphi_{- t}( Z_{1})\) is within the allowed region of motion, belongs to the level set  \((e_{1},e_{2})\in\mathcal{R}^{c}(h)\), and consists of i.c. which impact at time \(t\) the 1-boundary of the step transversely, with horizontal velocity \(-\sqrt{2(e_{1}-h_{1}^{step})}\). Similarly, defining  \(Z_{2}=\{z|z=\left(q_1,q_2^{wall},\pm\sqrt{2(e_{1}-V_{1}(q_1))},-\sqrt{2(e_{2}-h_{2}^{step})}\right),q_1\in (q_1^{min}(e_{1}),q_{1}^{wall})\}\),  the set  \(\varphi_{- t}( Z_{2})\) is within the allowed region of motion for sufficiently small \(t\) and consists of i.c. belonging  to the level set  \((e_{1},e_{2})\in\mathcal{R}^{c}(h)\) which impact at time \(t\) the 2-boundary  of the step transversely, with vertical velocity \(-\sqrt{2(e_{2}-h_{2}^{step})}\).   \end{proof}

While, for most cases ("non-resonant"), each trajectory  belonging to level sets  \((e_{1},e_{2})\in\mathcal{R}^{c}(h)\) hits both boundaries of the step many times, in some resonant cases, it is possible to have families of trajectories belonging to level sets  \((e_{1},e_{2})\in\mathcal{R}^{c}(h)\) that  hit only one of the step boundaries or even avoid collisions (resonant trajectories belonging to the interval \(J_{K}\) of (\ref{eq:explicitiemap}) with \(K=0\), see section \ref{sec:returnmaps} for more details).

\textbf{Action angle coordinates and transverse sections: }

   The action angle coordinates of the 1 d.o.f. Hamiltonian, \(H_{i}(q_i,p_i)\),  \((I_{i},\theta_i(t)=\omega_i(I_i)t+\theta_i(0))\),  are uniquely defined up to a shift in the angle. Since, by our assumptions, \(H_i(I_i)=e_i\) is invertible,   \(e_i\) or \(I_i\) may be used to label level sets (to simplify notation, we hereafter consider the frequencies as functions of the energies, \(e_{i}\)).
 By the monotonicity of \(V_{i}(q_{i})|_{q_{i}\neq0}\), for all energy surfaces \(h=e_1+e_2>0,\)  each energy surface contains a family of invariant tori on which rotations occur, and its boundary consists of the two invariant circles that correspond to the normal modes - the oscillatory motion of only one oscillator with the other one at rest
 (\(e_{1}=0,e_2=h\) and  \(e_{1}=h,e_2=0)  \).

\begin{figure}
\begin{centering}
\includegraphics[scale=0.45]{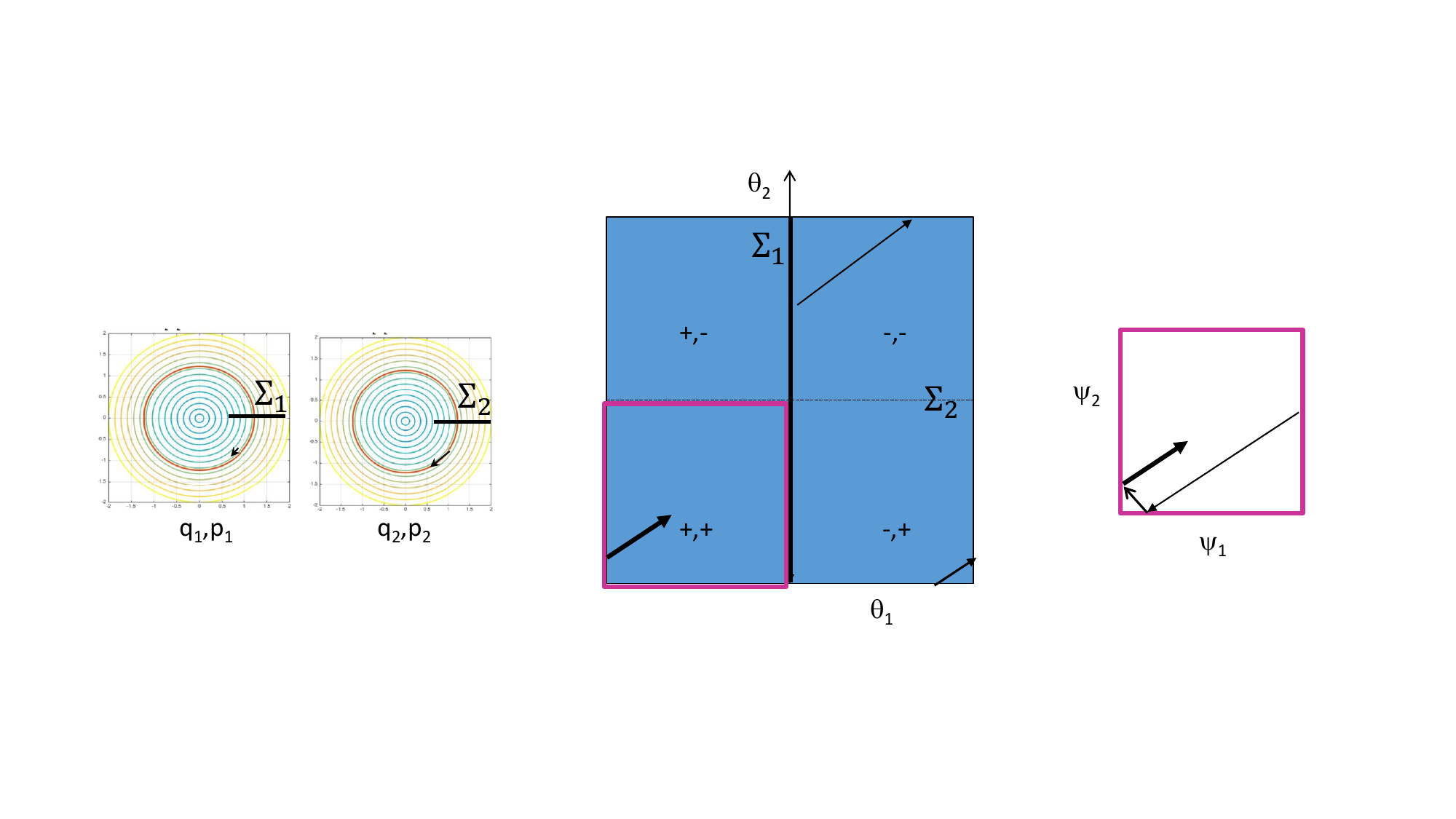}
\par\end{centering}
\protect\caption{\label{fig:signp1reg} Folding the smooth flow to a billiard: the motion on a level set is conjugated via action angle coordinated to the directional motion on the angles-torus. The motion is conjugated to the directional billiard motion on the left lower square. The direction of motion in this billiard is in the same quadrant as  the direction of motion in the configuration space (see Eq. (\ref{eq:signpsigntheta})).}
\end{figure}

For \(e_{i}>0\), denote by \(\Sigma_i\) the  three dimensional transverse section \(\{p_{i}=0,\dot p_{i}<0\},\)  and we set the phases of the action-angle coordinates to vanish  on these sections (so \(\theta_{i}=0\) (mod \(2\pi\)) on \(\Sigma_{i}\)):\begin{equation}
\Sigma_{i}: \{(q_{i},q_{\bar i},p_{i},p_{\bar i})|p_{i}=0, \dot p_{i}<0\}=\{(\theta_{i},\theta_{\bar i},I_{i},I_{\bar i})|\theta_{i}=0, I_i>0\}.\label{eq:defsigmai}
\end{equation}     By the symmetry of the mechanical Hamiltonian, with this choice of the phases, \(p_{i}(t)>0\) for \(\theta_{i}(t)\in(-\pi,0)\) (mod \(2\pi\)) and similarly \(p_{i}(t)<0\) for \(\theta_{i}(t)\in(0,\pi )\) (mod \(2\pi\)), namely \(
\text{sign}(p_{i}(t))=\text{sign}(\dot q_{i}(t))=-\text{sign}(\theta_{i}(t) \text{ mod }2\pi)\). For \(p_{i}\) which is bounded away from zero, the smooth flow is smoothly conjugate, through the action angle transformation, to the directional motion on the flat  torus in the direction \((\omega_1(e_1),\omega_2(e_2))\). The directed motion on the torus is  conjugated, by standard folding, to the directed billiard motion on the   square \((\psi_{1},\psi_2)\in[-\pi,0]\times[-\pi,0]\) (see figure \ref{fig:signp1reg}). For this specific folding and for the choice of the angle phase (\ref{eq:defsigmai}),  the direction of time is preserved along trajectories of the smooth flow and the billiard: \begin{equation}
\text{sign}(p_{i}(t))=\text{sign}(\dot q_{i}(t))=\text{sign}(\dot \psi_{i}(t))\label{eq:signpsigntheta}
\end{equation}
namely, the directed billiard in the square (hereafter called the \(\psi\)-billiard) and the smooth flow on the level set \((e_{1},e_2) \) are topologically conjugated, see Fig \ref{fig:signp1reg}. By reflections and time reversal, the flow is also conjugated to the billiard on the positive quadrant.

   We use the same construction of conjugacy  for the impact system.
 Let: \begin{equation}\label{def:sigmaplusminusqwall}
\Sigma_{i}^{\pm}=\{(q,p)|q_{i}=q_i^{wall},\pm p_{i}>0\},
\end{equation}
and let \(t_{\Sigma_{i}^{-}\rightarrow\Sigma_{i}^+}=T_{i}(e_{i}^{wall})-\tilde T_{i}(e_{i};q_i^{wall}),\,t_{\Sigma_{i}^{+}\rightarrow\Sigma_{i}}=t_{\Sigma_{i}\rightarrow\Sigma_{i}^-}
=\frac{1}{2}\tilde T_{i}(e_{i};q_i^{wall})\) denote the respective travel times between the sections.

\begin{lem}\label{lem:reflectionsangles} The sections \(\Sigma_{i}^{\pm}\)  are impacted/crossed transversely by the step-flow if and only if \(e_{i}>h_i^{step}\). For all i.c. belonging to a level set \(e_{i}>h_i^{step}\), with the angle coordinate convention (\ref{eq:defsigmai}), the angle \(\theta_{i}\)  at the section \(\Sigma_{i}^{-}\) is \(\theta_i^{wall}
(e_{i}) \):\begin{equation}\label{eq:thetaiwallexplicit}
\theta_i^{wall}
(e_{i};q_{i}^{wall})=\omega_i(e_i)t_{\Sigma_{i}\rightarrow\Sigma_{i}^-}= \omega_i(e_i)\int^{q_{i}^{max}(e_i)}_{q_i^{wall}} \frac{dq}{\sqrt{2(e_{i}-V_i(q_i))}}=\pi\frac{\tilde T_{i}(e_{i};q_i^{wall})}{T_{i}(e_{i};q_i^{wall})} ,
\end{equation}and a reflection from the step at \(q_i^{wall}\)    sends the angle  \(\theta_i^{wall}
(e_{i};q_{i}^{wall})\) to  \(2\pi-\theta_i^{wall}
(e_{i};q_{i}^{wall})=-\theta_i^{wall}
(e_{i};q_{i}^{wall}) \) mod \(2\pi  \). \end{lem}
\begin{proof}
Since the level sets of \(H_{i}\) are nested,   for   \(e_{i}<h_i^{step}\) the \(e_{i}\) level set is strictly interior to the \(h_i^{step}\) level set, and hence the sections \(\Sigma_{i}^{\pm}\)  are not reached by the flow. Conversely, for \(e_{i}>h_i^{step}\), the   sections \(\Sigma_{i}^{\pm}\)  are crossed by the level set, and, by the mechanical form of the Hamiltonian \(H_{i}\), on these sections  \(p_i^2=2(e_i-V(q_i^{wall}))>0\) so they are crossed transversely. The formula for \(\theta_i^{wall}
(e_{i};q_{i}^{wall}) \)  follows from the definition of action-angle coordinates and the convention (\ref{eq:defsigmai}). By the symmetry \(p_{i}\rightarrow-p_i\) of mechanical Hamiltonian function it follows that the reflection from the step at \(q_i^{wall}\)   sends the wall angle coordinate \(\theta_i^{wall}
\) to \(2\pi-\theta_i^{wall}
(e_{i})=-\theta_i^{wall}
(e_{i}) \) mod \(2\pi  \).    \end{proof}
 Notice that, as summarized in Table \ref{tab:thetat12},\begin{equation}\label{eq:thetailimits}
\lim_{e_{i}\searrow h_i^{step}}\theta_i^{wall}
(e_{i};q_{i}^{wall})=\begin{cases}\pi  & \text{for }q_i^{wall}<0 \\
0   & \text{for }q_i^{wall}>0. \\
\end{cases}
\end{equation}
and
\begin{equation}\label{eq:thetailimitinfinity}
\lim_{e_{i}\rightarrow\infty}\theta_i^{wall}
(e_{i};q_{i}^{wall})=\theta_i^{wall,\infty},
\end{equation}
where, for symmetric potentials,
\(\theta_i^{wall,\infty}=\frac{\pi }{2}\).

Combining  the classification of level sets according to their impacts with the boundaries (lemmas \ref{lemm:r12c}-\ref{lem:rcbothsides}) with the action-angle representation of the flow and the impacts on a given level set (lemma \ref{lem:reflectionsangles}), we  establish  the topological conjugacy between the impact flow on a given level set and its corresponding flat surface and billiard table.
To this aim,  it is convenient to define:

 \begin{equation}
\hat \theta_{ i}^{wall}(e_{i},e_{\bar i};q_{i}^{wall},q_{\bar i}^{wall})=\begin{cases}
\emptyset\ & \text{if }q_{1,2}^{wall}>0\wedge e_{ 1,2}<h_{ 1,2}^{step} \\
\theta_{ i}^{wall}(e_{ i};q_{ i}^{wall}) & \text{if }   e_{ i}\geq h_{ i}^{step}\wedge(e_{\bar i}\geq h_{\bar i}^{step}\vee q_{\bar i}^{wall}>0)\\
\pi \ & \text{otherwise. }
\end{cases}
\label{eq:thetahatforall}\end{equation}By lemmas \ref{lemm:r12c}-\ref{lem:rcbothsides}, \(\hat \theta_{ i}^{wall}(e_{i},e_{\bar i};q_{i}^{wall},q_{\bar i}^{wall})=\theta_{ i}^{wall}(e_{ i};q_{ i}^{wall})\) for level sets for which impacts
(transverse or tangent) with the \(i\)-boundary are allowed,  \(\hat \theta_{ i}^{wall}(e_{i},e_{\bar i};q_{i}^{wall},q_{\bar i}^{wall})=\emptyset\)  for level sets that are not in the allowed region of motion, and \(\hat \theta_{ i}^{wall}(e_{i},e_{\bar i};q_{i}^{wall},q_{\bar i}^{wall})=\pi\) for level sets in which impacts with the \(i\)-th boundary cannot occur.

\

\textbf{Rotational dynamics for level sets in \(\mathcal{\bar R}^{i}(h)\):}
\begin{lem}\label{lem:ridyntorus}
{  For level sets \((e_{1},e_2)\) in \(\mathcal{\bar R}^{i}(h)\) the step-dynamics are smoothly conjugate to the directional motion \((\omega_i(e_i),\omega _{\bar i}(e_{\bar i})) \) on the  torus \begin{equation}
\mathbb{T}_{i}(e_{1},e_{2})=\{(\theta_{ i},\theta_{\bar i})|\theta_{ i}\in[-\pi,\pi),\theta_{\bar i}\in[-\hat \theta_{\bar i}^{wall},\hat \theta_{\bar i}^{wall})\},
\label{eq:striptorus}\end{equation}
with \(\hat \theta_{\bar i}\) defined by (\ref{eq:thetahatforall}). This step-dynamics are also conjugated to the \((\pm\omega_i(e_i),\pm\omega _{\bar i}(e_{\bar i})) \) directional billiard motion   on the rectangular billiard  \((\psi_{ i},\psi_{\bar i})\in[-\pi,0]\times[-\hat \theta_{\bar i}^{wall},0]\).   In particular, the conjugation keeps the direction of motion: the signs of \(\dot \psi_{1,2}\) and the sign of \(\dot q_{1,2}\) coincide.  }     \end{lem}
\begin{proof}
By lemma \ref{lem:ridyn} the motion on level sets in  \(\mathcal{R}^{i}(h)\) is either: a) not defined (so   \(\hat \theta_{\bar i}^{wall}=\emptyset\)), b) corresponds to reflections only from the  \(\bar i\)- boundary of the step, or, c) the trajectory does not touch the step, so the motion occurs as in the non-impact case on the torus  (\ref{eq:striptorus}) with     \(\hat \theta_{\bar i}^{wall}=\pi\).

The three rows of conditions in the definition (\ref{eq:thetahatforall}) of      \(\hat \theta_{\bar i}^{wall}\) for \(e_{ i}<h_{ i}^{step}\) coincide with the conditions listed for cases a,b,c in  lemma \ref{lem:ridyn}, so, to complete the proof we only need to show that case b) indeed corresponds to the rotation on the clipped torus  (\ref{eq:striptorus}) with     \(\hat \theta_{\bar i}^{wall}=\theta_{\bar i}^{wall}\).    Indeed, by the mechanical form of \(H_{\bar i}\),   reflections only from the  \(\bar i\)- boundary of the step imply that  the corresponding angle coordinate is restricted to the interval   \(\theta_{\bar i}(t)\in[-\theta_{\bar i}^{wall}(e_{\bar i};q_{\bar i}^{wall}),\theta_{\bar i}^{wall}(e_{\bar i};q_{\bar i}^{wall})]  \), where,  by lemma \ref{lem:reflectionsangles}, the transverse impacts  correspond to gluing the transverse section  \(\Sigma_{\bar i}^{\pm}|_{H_{\bar i}=e_{\bar i}}\):    \begin{equation}
\Sigma_{i}^{-}|_{H_i=e_i}=\{(\theta,I)| I_i= I_i(e_i),\theta_{i}=\theta_i^{wall}
(e_{i})\}, \:\:\Sigma_{i}^{+}|_{H_i=e_i}=\{(\theta,I)| I_i= I_i(e_i),\theta_{i}=-\theta_i^{wall}
(e_{i})\}.\end{equation}  by identifying the angles \(\theta_{\bar i}^{wall}(e_{\bar i};q_{\bar i}^{wall})\) and \(-\theta_{\bar i}^{wall}(e_{\bar i};q_{\bar i}^{wall})  \). Namely, we obtain a directional motion on the torus (\ref{eq:striptorus}),  in the direction \((\omega_i(e_i),\omega_{\bar i}(e_{\bar i})) \). By folding to the rectangle \((\psi_{ i},\psi_{\bar i})\in[-\pi,0]\times[-\hat \theta_{\bar i}^{wall},0]\), the motion is conjugated to the \(\psi\)-billiard in this rectangular billiard, and (\ref{eq:signpsigntheta}) holds for the impact flow as well, proving the lemma for this case as well, see tables IIA,IIIIA,IIID,IIIID of Fig. \ref{fig:lshapesall}.
 \end{proof}

\begin{figure}
\begin{centering}
\includegraphics[scale=0.6]{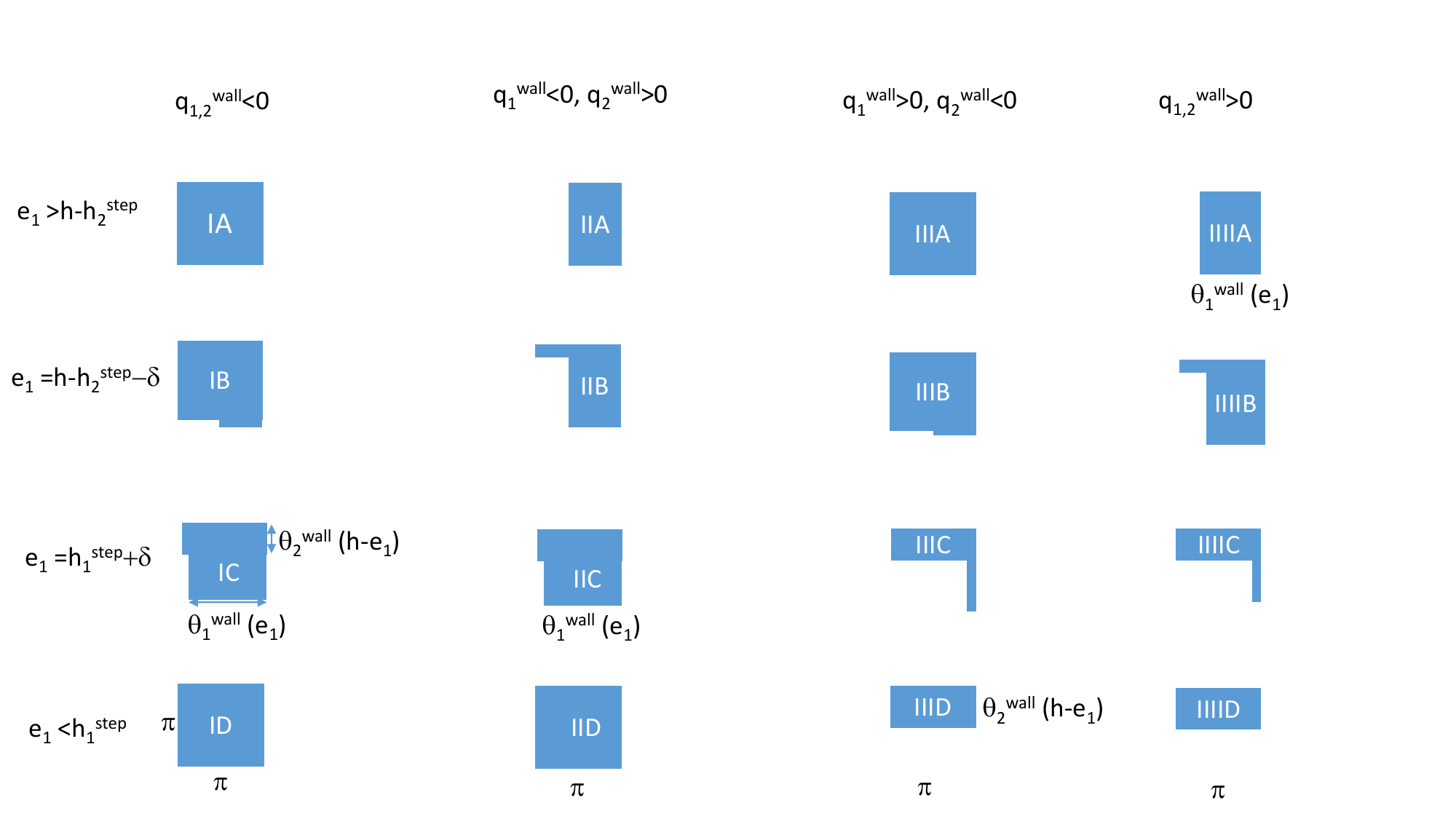}
\par\end{centering}
\protect\caption{\label{fig:lshapesall} The iso-energy   billiard geometry at the different step locations for \(h>h^{step}\). The first and last rows present, respectively, the rectangular billiards for level sets in \(\mathcal{R}^{2}(h)\) and  \(\mathcal{R}^{1}(h)\). The second and third rows present, respectively, the L-shaped billiards in  \(\mathcal{R}^{c}(h)\) just below and just above the edges of the  \(\mathcal{R}^{c}(h)\) interval  (so \(\delta>0\) is  small). }
\end{figure}

\textbf{The flow in the region \(\mathcal{R}^{c}(h)\)  is conjugated to the L-shaped billiard flow:}

\begin{figure}
\begin{centering}
\includegraphics[scale=0.4]{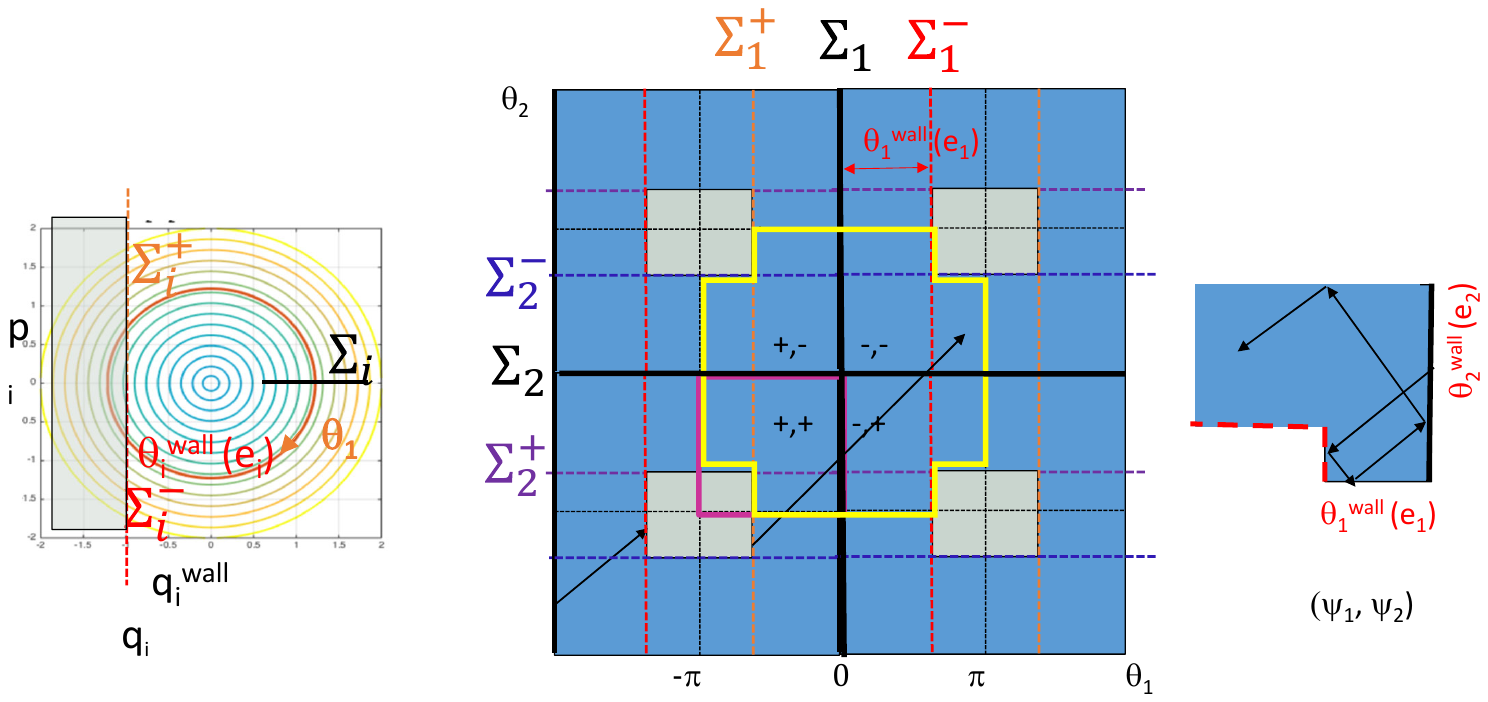}
\par\end{centering}
\protect\caption{\label{fig:cross} The step return map, the swiss cross surface and the rotated L-shaped billiard geometry for level sets in the step region, \(\mathcal{R}^{c}(h)\). The grey areas correspond to the step region in the angles space. The yellow outlines the boundary of SW, the Swiss-cross flat surface for which opposite parallel sides are glued.} \end{figure}

\begin{lem}\label{lem:lshaped}
{  For level sets \((e_{1},e_2)\) in \(\mathcal{R}^{c}(h)\) the step-dynamics are  conjugate to the directional motion \((\omega_1(e_1),\omega _{2}(e_{2})) \) on SW - the  swiss-cross-shaped \((\theta_{1},\theta_{2})\)-surface with vertical arms of width  \(2\theta_{1}^{wall}(e_{1})\)  and length \(2\pi,\)  horizontal arms of height  \(2\theta_{2}^{wall}(e_{2})\) and width  \(2\pi\) and the flat surface is achieved by gluing of parallel opposite sides. This step-dynamics are also conjugate to the \((\pm\omega_i(e_i),\pm\omega _{\bar i}(e_{\bar i})) \) directional billiard motion   on the L-shaped billiard    \(L(\pi,\pi,\theta_{1}^{wall}(e_{1};q_{1}^{wall}),\theta_{2}^{wall}(h-e_{1};q_{2}^{wall}))\). Reflecting the L-shaped billiard  with respect to the \(\theta_1\)-axis and the \(\theta_2\)-axis provides dynamics with conjugation that keeps the direction of motion.      } \end{lem}
\begin{proof}
Recall that with the convention (\ref{eq:defsigmai}),  \(q_i(t;e_{i})>q_i^{wall}\) iff the angle coordinate of the smooth flow is in the interval  \((-\theta_{i}^{wall}(e_{i};q_i^{wall}),\theta_{i}^{wall}(e_{i};q_i^{wall}))\).  Hence, on a level set  \((e_{1},e_2)\in\mathcal{R}^{c}(h)\), the disallowed step region in the configuration space is mapped by the smooth action-angle transformation to a disallowed rectangular region in the angle variables: \((\theta_{1},\theta_{2})\in S_{\theta(e_{1},e_2)}:=[\theta_{1}^{wall}(e_{1};q_1^{wall}),2\pi-\theta_{1}^{wall}(e_{1};q_1^{wall})]\times[\theta_{2}^{wall}(e_{2};q_2^{wall}),2\pi-\theta_{2}^{wall}(e_{2};q_2^{wall})] \) all taken mod \(2\pi\). This rectangle cuts the four corners of the fundamental domain creating a swiss cross surface (see Fig. \ref{fig:cross}).
 By lemma \ref{lem:reflectionsangles},  the reflection  rule at impact, \(p_i\rightarrow-p_i\), translates to \(\theta_{i}^{wall} \rightarrow2\pi-\theta_{i}^{wall} \). Hence, the resulting flow under the step dynamics, expressed in the smooth action angle coordinates, corresponds to setting the action values to constants, \(I_{i}(e_{i})\), and letting the angles \((\theta_{1},\theta_2)\) increase linearly at constant speeds \((\omega_1(e_1),\omega_2(e_2))\) on the  torus \([0,2\pi]\times[0,2\pi]\), till the rectangle \(S_{\theta(e_{1},e_2)}\) is met. There, the gluing condition   \(\theta_{i}^{wall}(e_{i};q_i^{wall})\rightarrow2\pi-\theta_{i}^{wall}\) is  applied. This is a directed flow on a "torus with a rectangular hole" namely, a compact  orientable surface of genus 2. Equivalently, when shifting the torus center by \((-\pi,-\pi)\), this is a directed flow on a swiss-cross surface, see Figure \ref{fig:cross}. For all   \((e_{1},e_2=h-e_{1})\in\mathcal{R}^{c}(h)\), the dynamics under this gluing rule of the swiss-cross correspond to an unfolding of a billiard motion in the \(\mathbf{B}(e_{1})=L(\pi,\pi,\theta_{1}^{wall}(e_{1};q_1^{wall}),\theta_{2}^{wall}(h-e_{1};q_2^{wall}))\)-shaped table  \cite{Athreya2012,Zorich2006} in the directions \((\pm\omega_1(e_1),\pm\omega_2(h-e_1))\), where, as before, by the choice (\ref{eq:defsigmai}) of the angle phases, (\ref{eq:signpsigntheta}) holds on the L-shaped billiard that is folded onto the  low-left part of the swiss cross (see Fig \ref{fig:cross} and Fig. \ref{fig:Billards/Regular}). Thus, we have shown that the dynamics on the iso-energetic level sets in   \(\mathcal{R}^{c}(h)\) are conjugated to the family of \(\alpha-\)directional flows on the family of L-shaped billiards,  \(\mathcal{B}(h)=\{\alpha (e_1)=\frac{\omega_2(h-e_1)}{\omega_1(e_1)},\mathbf{B}(e_{1})\}|_{e_{1}\in\mathcal{I}^{c}(h)}\), where
 \(\mathcal{I}^{c}(h):=(h_{1}^{step}, h- h_{2} ^{step}=h_{1}^{step}+h-h^{step})\).
\end{proof}

\begin{figure}
    \centering
    \includegraphics[scale = 0.3]{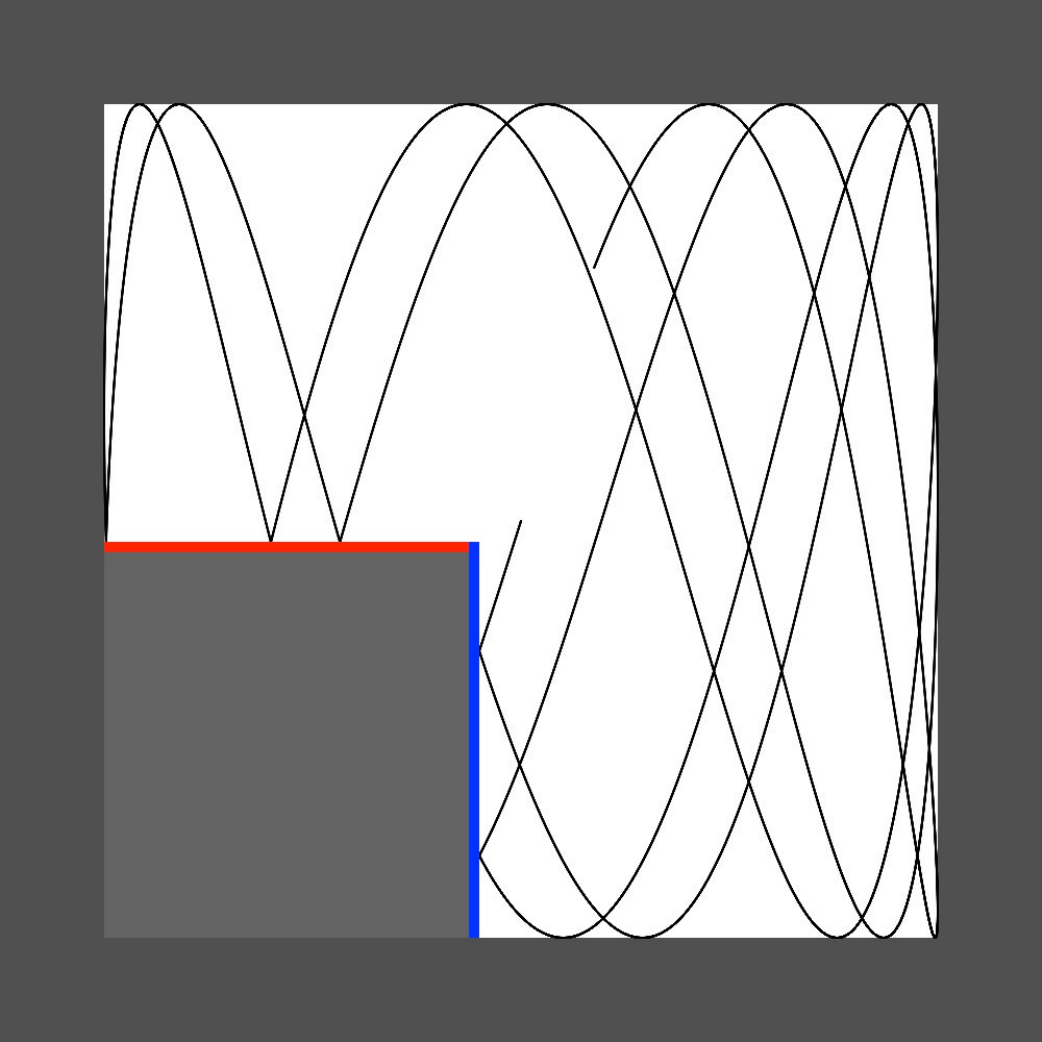}\qquad \qquad \qquad
    \includegraphics[scale = 0.3]{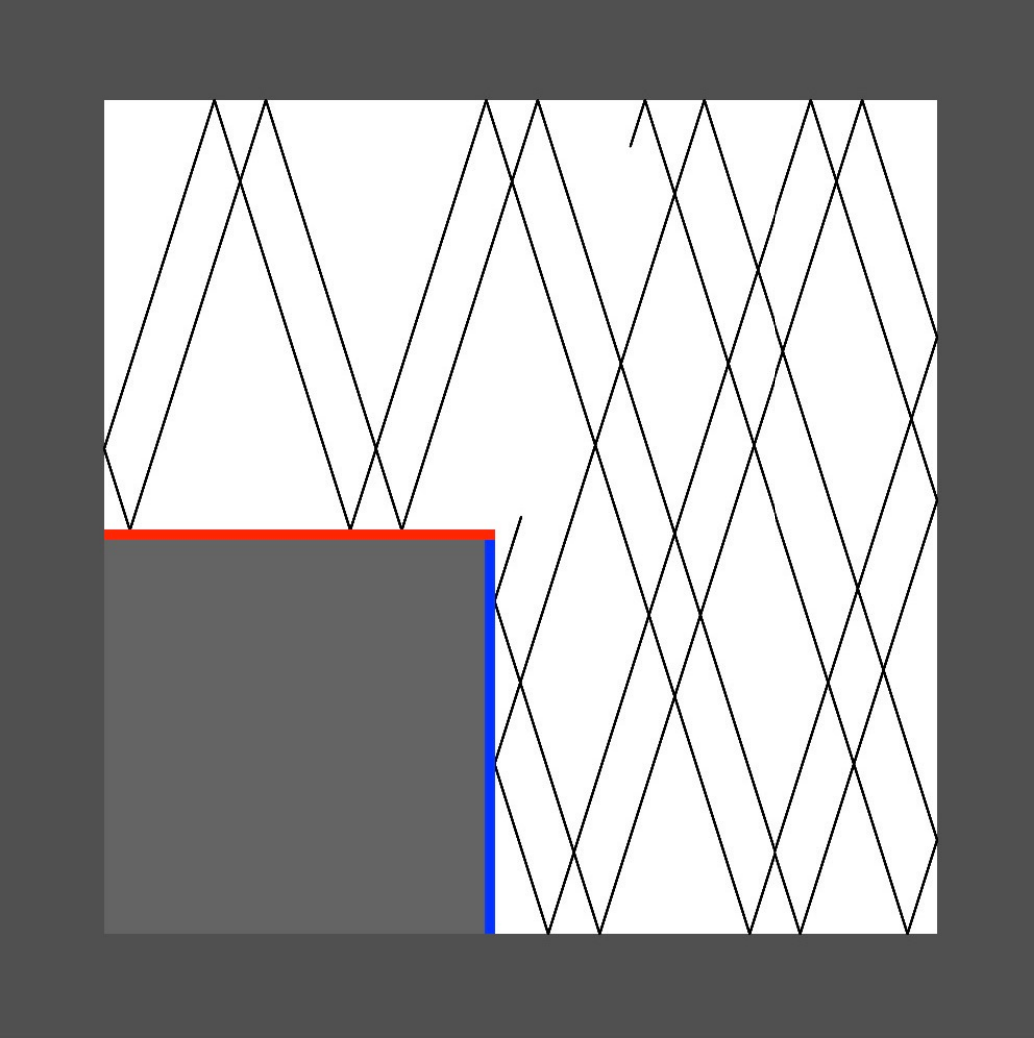}
    \caption{A simulation of the configuration space of the linear oscillators step system (left) with its corresponding matching L-shaped billiard in the angle space (right). The turning points of the flow, where \(p_{i}=0\), are mapped to reflections from the outer square boundaries and the elastic reflections of the flow from the step are mapped to the billiard reflections from the step.}
    \label{fig:Billards/Regular}
\end{figure}

 Finally, to complete the proof of Theorem \ref{thm:mainflow}, we notice that  since the directed flow  on a genus-2 orientable compact surface is not conjugate to a flow on a torus, and since by lemma \ref{lem:lshaped}  the motion on the level sets \((e_{1},h-e_2)\) for all \({e_{1}\in\mathcal{I}^{c}(h)}\)  is conjugated to such a flow, the step system is not LIHIS. The measure of the corresponding set is positive as the intersection of each level set in  \(\mathcal{R}^{c}(h)\) with the allowed region of motion has positive area and   \(|\mathcal{I}^{c}(h)|=h-h^{step}>0\).  By Lemmas \ref{lem:ridyntorus} the motion on the iso-energy level sets   \((e_{1},h-e_{1})\) with  \(e_{1}\in (0,h_1^{step})\cup( h- h_{2} ^{step},h)\), the iso-energy complement to \(\mathcal{R}^{c}(h)\), is conjugate to the directed flow on a torus, and this complement also has positive measure since, for \(h>h^{step}\), the intersection of these level sets with the allowed region of motion is always of positive measure.

Each column of Fig. \ref{fig:lshapesall} shows  schematically the family of iso-energetic billiard tables obtained for  the indicated positions of the step.
 The directional  L-shaped billiard families,  \(\mathcal{B}(h)\), are shown in rows B and C and correspond to level sets in \(\mathcal{R}^{c}(h)\). The widths of the arms of L-shaped tables at the edges of the segment  \(\mathcal{R}^{c}(h)\) (these depend on the signs of \(q_{1,2}^{wall}\)) are listed in Table \ref{tab:thetat12} - note that they are distinct, namely, for all \(h>h^{step}\), \(\theta_{i}^{wall}(h_{i}^{step})\neq\theta_{i}^{wall}(h-h_{\bar i}^{step})\).  The rectangular billiards shown in rows A and D correspond to level sets in  \(\mathcal{R}^{1}(h)\) and   \(\mathcal{R}^{2}(h)\) respectively.

Lemma \ref{lem:reflectionsangles}  in the above proof exposes the simple relation between reflections from vertical and horizontal boundary segments and the corresponding gluing rule in the angles variables. Corollaries  \ref{cor:nontrivialtop} and \ref{cor:moresteps} follow from this construction;  steps (two rays meeting at a \(\frac{3\pi}{2}\) corner) produce for sufficiently high individual energies a single hole, a staircase in the configuration space creates at sufficiently high individual energies a nibbled hole in the angles variables, a strip with handles creates, for intermediate  individual energies several disconnected components and for  sufficiently high individual energies  two holes, and a rectangle creates for sufficiently high individual energies four holes, see Fig. \ref{fig:torus1} for a demonstration. Thus, by constructing a nibbled scattering geometry which combines finite and semi-infinite horizontal and vertical segments in the configuration space, the number of holes and the number of connected components in the iso-energy level set surfaces can be  manipulated. Moreover, constructing an impact energy-momentum diagram  \cite{pnueli2018near,Pnueli2019}, such as Fig \ref{fig:embd} for the one-step system, allows to identify the critical energy values at which the topology of the energy surface changes.

\begin{table}[ht]
\begin{center}
  \begin{tabular}{|c|c|c|}\hline
Corner position\ & \(\theta_{1}^{wall}(h_{1}^{step})\) &   \(\theta_{2}^{wall}(h_{2}^{step})\) \\\hline\hline
\(q_{1,2}^{wall}<0\) & \(\pi\) &\( \pi\) \\\hline
 \(q_{1}^{wall}<0\), \(q_{2}^{wall}>0\) & \(\pi\)  & \(0\) \\\hline
 \(q_{1}^{wall}>0\), \(q_{2}^{wall}<0\) & \(0\)  & \( \pi \) \\\hline
\(q_{1,2}^{wall}>0\) & \(0\)  &\(0\) \\\hline
\end{tabular}
\end{center}
\caption{The  values of \(\theta_{1,2}^{wall}\) at the two edges of \(\mathcal{R}^{c}(h)\). The values of    \(\theta_{2}^{wall}(h-h_{1}^{step};q_2^{wall})\) and \(\theta_{1}^{wall}(h-h_{2}^{step};q_1^{wall})\) vary accordingly with \(h \), with limiting values  \(\theta_i^{wall,\infty}\in(0,\pi)\), see rows B,C of Fig. \ref{fig:lshapesall}. }\label{tab:thetat12}
\end{table}

\begin{figure}
    \includegraphics[scale = 0.5]{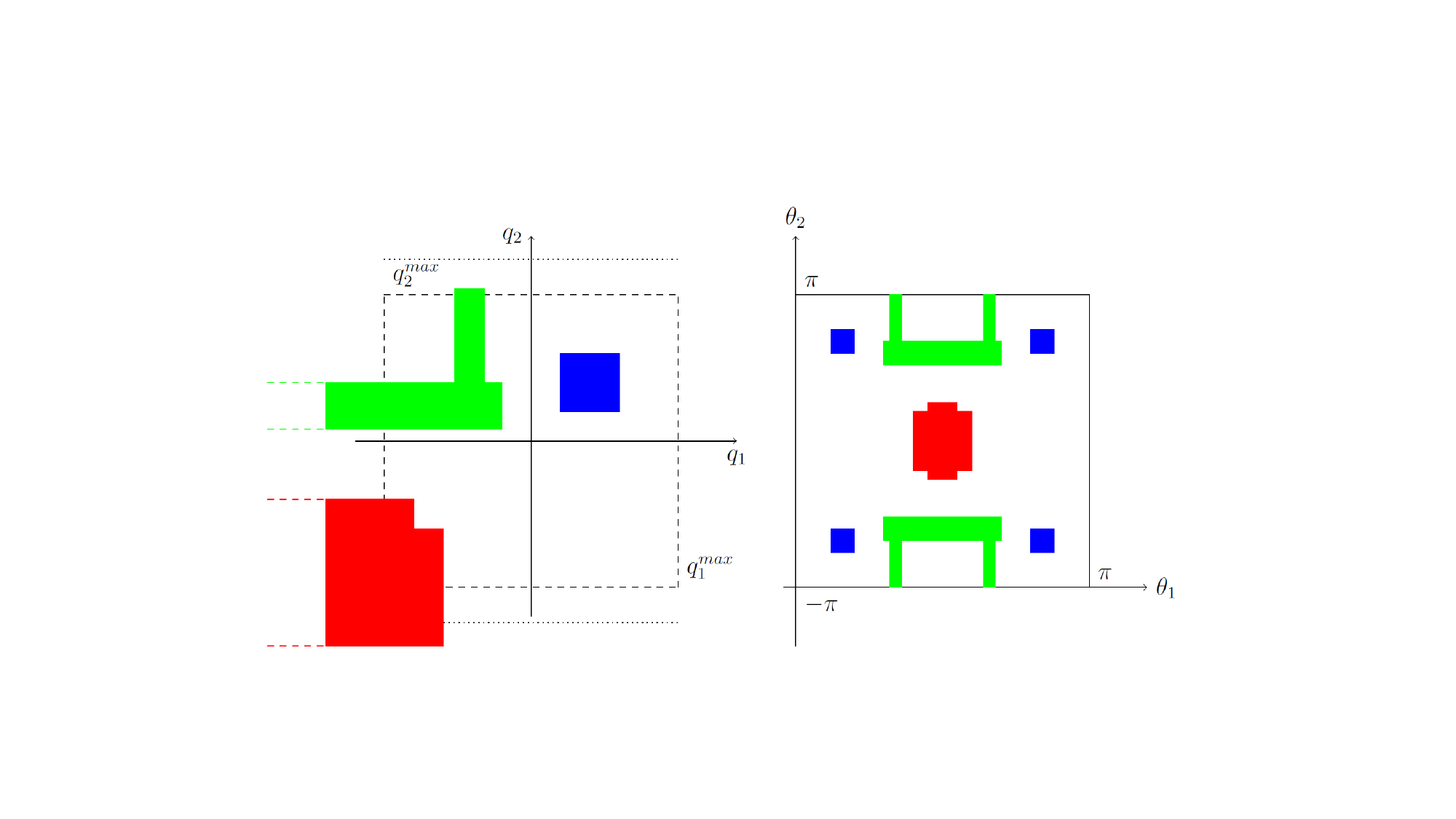}
\caption{For the indicated level set (dashed line), a 2-step staircase (red), a strip with a handle (green) and a block (blue) in the configuration space (left figure) create, respectively, one, one and four holes in the angle-angle torus representation, and divide the torus to two disconnected components (inside and outside of the green frame). A slight increase in the vertical energy \(e_{2}\) (dotted lines) makes the level set surface connected with two green holes.  }
\label{fig:torus1}
\end{figure}

 \section{Return maps\label{sec:returnmaps} }

 \textit{Proof of Theorem \ref{thm:mainIEMsh}}  In Theorem \ref{thm:mainflow} we proved that the step dynamics on each iso-energy level set is conjugated, via the action angle transformation, to the  \((\omega_1(e_1),\omega _{ 2}(h-e_{1})) \) directional flow on a flat surface - a glued swiss cross for level sets in \(\mathcal{R}^{c}(h)\) (lemma \ref{lem:lshaped}) and  a torus for level sets in the complement to \(\mathcal{R}^{c}(h)\)  (lemma \ref{lem:ridyntorus}). The transverse Poincar\`e section  \(\Sigma_{1}\) of the step flow is conjugated to the transverse section  \(\theta_{1}=0\) on these surfaces via the action-angle transformation  (recall (\ref{eq:defsigmai}), and notice that the assumptions on the potentials imply that \(\omega_1(e_1)\) is bounded away from zero for any finite \(e_{1}\)), so the return map of the step flow to  \(\Sigma_{1}\) is conjugated to the return map to  \(\Sigma_{1}\) on the corresponding flat surface. The return map to \(\Sigma_{1}\) on the flat surface is an interval exchange map on a circle:  for the swiss cross a three-interval exchange map and for the torus a rotation of a single interval (see, e.g. \cite{Zorich2006}). For a fixed fundamental interval on this circle, the return map becomes, in general, a 5-IEM for the swiss-cross case and a 2-IEM for the torus. Computations of the resulting IEMs (see Theorems \ref{thm:mainmap}) show that the lengths of the intervals of the 5-IEMs and their positions on the circle for  iso-energy level sets change smoothly in the step region. In particular, conditions for having a zero length interval are expressed as an equation of smooth, non-constant functions of \(e_{1}\) which are shown to vanish at most at isolated \(e_{1}\) values in the interior of  \(\mathcal{R}^{c}(h)\) .

 Next, we calculate the iso-energetic family of IEMs,  \(\mathcal{F}(h)=\{F=F_{(e_{1},h-e_1)}\}_{e_1\in[0,h]}\)   for the 2-IEM case (Theorem \ref{thm:mainmapR12}) and for the  5-IEM case (Theorem \ref{thm:mainmap}) thus completing the proof of Theorem \ref{thm:mainIEMsh}. In section \ref{sec:moreproperties}
we explore some of the properties of the 5-IEM\ family.\  \\ \\
Let \(\Theta_{2}\) denote the gain in the  \(\theta_{2}\) phase of the return map to \(\Sigma_{1}\) when the motion is to the right of the step:\begin{equation}
\Theta_{2}=\Theta_{2}(e_{1},h;q_{1,2}^{wall})=\frac{\hat\theta_{ 1}^{wall}}{\pi}\Theta^{smooth}_{2}=\begin{cases}2\pi\frac{ \tilde T_{1}(e_1;q_{1}^{wall})}{T_{ 2}(h-e_{1})}  &  \text{if }\hat \theta_{ 1}^{wall}(e_{1},h;q_{1,2}^{wall})\neq\pi\ \\
\Theta_{2}^{smooth}(e_{1},h) & \text{if }\hat \theta_{ 1}^{wall}(e_{1},h;q_{1,2}^{wall})=\pi, \\
\end{cases}\label{eq:Theta}
\end{equation}where \(\hat \theta_{ 1}^{wall}(e_{1},h;q_{1,2}^{wall})\) (see   Eq. (\ref{eq:thetahatforall})) is the effective impact angle with the side boundary of the step and  \(\Theta_{2}^{smooth}(e_{1},h)\) (see Eq. (\ref{eq:Theta1smooth})) is the rotation in \(\theta_{2}\) for non-impacting trajectories.  Notice that for all level sets on which motion is defined \( \Theta_{2}\leq\Theta_{2}^{smooth}\). Let\begin{equation}
\Theta_{2}^{*}(e_{1},h;q_{1,2}^{wall})=2\hat \theta_{ 2}^{wall}\left\{\frac{\Theta_{2}}{2\hat \theta_{ 2}^{wall}}\right\}\end{equation}where \(\left\{x\right\}\)  denotes hereafter the fractional part of the number \(x\). We first establish that in the complementary sets to \(\mathcal{R}^{c}(h)\) the return map  to \(\Sigma_{1}\) is the rotation (\ref{eq:theta2rotation}):
\begin{thm}
\label{thm:mainmapR12}  Under the same conditions of Theorem \ref{thm:mainflow},  for all iso-energy level sets in   \(\mathcal{R}^{1}(h)\cup \mathcal{R}^{2}(h)\), the return map \(F_{(e_{1},h-e_1)}\) to the section \(\Sigma_{1}\)  is  topologically conjugated to a   \(\Theta_{2}\) rotation on the  \([ -\hat \theta_{ 2}^{wall},\hat \theta_{ 2}^{wall}) \) circle:\begin{equation}\label{eq:theta2rotation}
\theta_{2}\rightarrow\theta_{2}+\Theta_{2}(e_{1},h;q_{1,2}^{wall})\text{ mod }  2\hat\theta_{2}^{wall},
\end{equation}
or, equivalently, to a 2-IEM on the interval   \([ -\hat \theta_{ 2}^{wall},\hat \theta_{ 2}^{wall}] \)    with intervals lengths
 \(\lambda_{A}=2\hat\theta_{2}^{wall}-\Theta_{2}^{*},\lambda_{B}=\Theta_{2}^{*}\).\end{thm}
 \begin{proof} By lemma \ref{lem:ridyntorus} the flow on  level sets belonging to      \(\mathcal{R}^{1}(h)\)  is  topologically conjugated to the \((\omega_1(e_1),\omega _{ 2}(h-e_{1})) \) directional flow on the torus \(\mathbb{T}_{1}(e_{1},h-e_{1})\) of Eq. (\ref{eq:striptorus}). Notice that if the level set is in the disallowed region of       \(\mathcal{R}^{1}(h)\), then   \(\hat \theta_{2}^{wall}=\emptyset\), hence \(\mathbb{T}_{1}(e_{1},h-e_{1})=\emptyset\), so the Theorem trivially holds. For the non-trivial case, by (\ref{eq:defsigmai}), the transverse section  \(\Sigma_{1}\)  to the flow is mapped, for a fixed level set, to the transverse section \(\theta_{1}=0\) of the corresponding torus. Hence, to complete the proof we need to show that the return map to the section \(\theta_{1}=0\) of the  \((\omega_1(e_1),\omega _{ 2}(h-e_{1})) \) directional flow on  \(\mathbb{T}_{1}(e_{1},e_{2})\)  is the rotation  (\ref{eq:theta2rotation}). Indeed, notice that for the level sets in      \(\mathcal{R}^{1}(h)\) the effective impact angle is \(\hat \theta_{1}^{wall}=\pi\) (when motion is allowed), so \( \Theta_{2}=\Theta_{2}^{smooth}(e_{1},h)=2\pi\frac{  \omega _{ 2}(h-e_{1})}{\omega _{ 1}(e_{1})}\) and thus   (\ref{eq:theta2rotation}) coincides with the return map on the \(\mathbb{T}_{1}(e_{1},h-e_{1})\) torus.  Similarly, by lemma \ref{lem:ridyntorus}, the flow on  level sets belonging to      \(\mathcal{R}^{2}(h)\)  is  topologically conjugated to the \(\)\((\omega_1(e_1),\omega _{ 2}(h-e_{1})) \) directional flow on the rotated torus \(\mathbb{T}_{2}(e_{1},e_{2})\) of Eq. (\ref{eq:striptorus}), namely on   \(
\mathbb{T}_{2}(e_{1},e_{2})=\{(\theta_{ 1},\theta_{2})|\theta_{1}\in[-\hat \theta_{1}^{wall},\hat \theta_{1}^{wall}),\theta_{2}\in[-\pi,\pi)\}.
\label{eq:striptorus1}\)
The return map to the section  \(\theta_{1}=0\) on this torus is  a rotation of the \(\theta_{2}\) angle on the \(2\pi\) circle by \(\omega _{ 2}(h-e_{1})\frac{2\hat \theta_{1}^{wall}}{\omega_1(e_1)}\), which is exactly \(\Theta_2\) (see Eq. (\ref{eq:Theta})). Finally, since  \(\hat \theta_{2}^{wall}=\pi\) for the allowed level sets in \(\mathcal{R}^{2}(h)\),   (\ref{eq:theta2rotation}) is verified.    \end{proof}
Next, we establish that for  level sets  in   \(\mathcal{R}^{c}(h)\),  the return map defines a three-interval map on the circle, namely a 5-IEM on the fundamental segment arises. Let

  \begin{equation}
\chi_{2}(e_{1},h;q_{1,2}^{wall})=\frac{\Theta_2^{smooth}-\Theta_2}{2\theta_{2}^{wall}}=\frac{ T_{1}(e_1)-\tilde T_{1}(e_1;q_{1}^{wall})}{\tilde T_{ 2}(h-e_{1};q_{2}^{wall})}=
\frac{\omega_{2}(h-e_{1})}{\omega_{1}(e_1)}\frac{\pi-\theta_{1}^{wall}(e_{1};q_1^{wall})}{\theta_{2}^{wall}(h-e_{1};q_2^{wall})} \label{eq:psi} \end{equation}
denote the ratio\ between the time  spent above the step and the return time to the upper step boundary. The integer part of \(\chi_{2}\) corresponds to the minimal number of impacts with the upper boundary of the step during this passage:
 \begin{equation}
K_{2}(e_{1},h;q_{1,2}^{wall})=\left\lfloor \chi_{2}\right\rfloor.\label{eq:kofe1}\end{equation}

\begin{thm}
\label{thm:mainmap}  Under the same conditions of Theorem \ref{thm:mainflow},  for all iso-energy level sets in   \(\mathcal{R}^{c}(h)\), the return map \(F_{(e_{1},h-e_1)}\) to the section \(\Sigma_{1}\)  is topologically conjugated to a 3 interval  IEM on the  \(\theta_{2}\) circle of the form:\begin{equation}\label{eq:intervalsorder}
 (J_{R},J_{K_{2}},J_{K_{2}+1})\rightarrow\Theta_{2}+(J_{R},J_{K_{2}+1},J_{K_{2}}) \text{ mod }2\pi,
\end{equation}
 where the lengths of the intervals are:\begin{equation}
(\lambda _{J_{R}},\lambda _{J_{K_{2}}},\lambda _{J_{K_{2}+1}})=(2\pi-2\theta_{2}^{wall},2\theta_{2}^{wall}(1-\left\{\chi_{2}\right\}),2\theta_{2}^{wall}\left\{\chi_{2}\right\}
),\label{eq:intervallengths}\end{equation}
  and the phase of the left boundary of \(J_{R}\) is\begin{equation}
\theta^{L}_{J_{R}}=\theta_{2}^{wall}-\frac{1}{2}\Theta_{2} \text{  (mod  $2\pi$).}
\label{eq:thetajrleft}\end{equation} In the above formulae  \((\theta_{ 2}^{wall}=\theta_{ 2}^{wall}(h-e_1;q_{ 2}^{wall}),\Theta_{2},\chi_{2})\) are defined by Eqs. (\ref{eq:thetaiwalldef},\ref{eq:Theta},\ref{eq:psi}) respectively and the phase \(\theta_{2}\) is set by (\ref{eq:defsigmai}). The return time to \(\Sigma_{1}\) for \(\theta_2\in J_{R}\) is \(\tilde T_{1}\) whereas for \(\theta_2\in J_{K_{2}}\cup J_{K_{2}+1}\) it is \( T_{1}\).       Equivalently, the dynamics for each level set is conjugated to the induced 5-IEM on the   \([ -\pi,\pi) \) interval of \(\theta_{2}\) values. This 5-IEM is uniquely defined  by Eqs. (\ref{eq:intervalsorder}-\ref{eq:thetajrleft}), and apart of isolated points of \(e_{1}\)   values in \(\mathcal{R}^{c}(h)\), all its 5 intervals are of positive lengths.
\end{thm}

\begin{proof}

 By lemma \ref{lem:lshaped} the flow on  level sets belonging to      \(\mathcal{R}^{c}(h)\)  is  topologically conjugated to the \((\omega_1(e_1),\omega _{ 2}(h-e_{1})) \) directional flow on SW - the swiss-cross surface defined by  \(\theta_{ 1}^{wall}(e_1;q_{1}^{wall})\) and \(\theta_{ 2}^{wall}(h-e_1;q_{ 2}^{wall})\). In particular, the section \(\Sigma_1\) of the return map is mapped by the action-angle conjugation to the vertical center of SW, the \(2\pi\) circle of \(\theta_{2}\) phases (see Fig. \ref{fig:cross}), so the return map on SW and the step dynamics return map to  \(\Sigma_1\) are smoothly conjugated. While the return map can be computed from the SW geometry alone, we find it convenient at times to  consider the step dynamics.\\ We divide the \(\theta_2\) circle to two sub-intervals: \(J_{R}\) consisting of phases with trajectories which hit the right boundary of the step (equivalently, the right boundary of the vertical arm of SW) and return to \(\Sigma_{1}\), and \(J_{U}\) consisting of phases with trajectories which do not hit the right boundary (equivalently, enter the horizontal arm of the SW), go above the step, possibly hitting the upper boundary of the step (equivalently, the horizontal boundaries of the SW horizontal arm), and then return to \(\Sigma_1\) (see Fig. \ref{fig:returnmap} and Fig. \ref{fig:cross} where the return map construction to \(\Sigma_{1}\) in the configuration space and in the directional flow on the swiss-crossed shaped polygon are shown). Hence, the length of \(J_{R}\) is the length of the vertical right boundary of the SW, \(\lambda _{J_{R}}=2\pi-2\theta_{2}^{wall} =2\pi(1-\frac{\tilde T_2}{T_{2}})\) and \(\lambda _{J_{U}}=2\theta_{2}^{wall} \).\\ The  return time for trajectories belonging to \( J_{R}\) is \(\tilde T_{1}\), the phase \(\theta_{2}\) for these trajectories increases at the constant speed \(\omega_{2}(h-e_1\)), so, the interval \( J_{R}\) is  rotated by
 \(2\frac{ \omega_{2}(h-e_1)}{\omega_{ 1}(e_{1})}\theta_{1}^{wall}\), namely by \(\Theta_{2}\) as defined in (\ref{eq:Theta}). \\  The  return time for trajectories belonging to \( J_{U}\) is  \( T_{1}\). It is divided to the time \(\tilde T_{1}\), where the trajectories are to the right of the step and to   the time interval \(T_{1}(e_{1})-\tilde T_{1}(e_{1};q_1^{wall})\) where the trajectories are above the step, possibly bouncing off its upper boundary. During the  \(\tilde T_{1}\) time segment the phase \(\theta_{2}\) increases, as before, by \(\Theta_{2}\). During the  \(T_{1}(e_{1})-\tilde T_{1}(e_{1};q_1^{wall})\) segment, the phase gain depends on the number of bounces. Denote the interval of \(\theta_{2}\) values for which trajectories hit the upper step \(k\) times by \(J_{k}\). The function \(\chi_{2}\) (see  Eq. (\ref{eq:psi})) provides the ratio between the time trajectories in \(J_{U}\) spend above the step and the return time of trajectories with energy \(e_{2}=h-e_1\)\ to the step upper boundary. Hence, the number of  bounces of the trajectories belonging to \(J_{U}\) is either   \(K_{2}=\left\lfloor \chi_{2}\right\rfloor\) or \(K_{2}+1\), namely,  \(J_{U}= J_{K_2}\cup J_{K_2+1}\). \\ The phase gained during the    \(T_{1}(e_{1})-\tilde T_{1}(e_{1};q_1^{wall})\) segment by trajectories in \(J_k\) is \(2(\pi-\theta_{2} ^{wall})k+\omega_{2}\tilde T_2\chi_{u}=2(\pi-\theta_{2} ^{wall})k+2\theta_{2}^{wall}(\left\lfloor\chi_{2}\right\rfloor+\left\{\chi_{2}\right\})=2\pi k+2\theta_{2} ^{wall}(K_{2}-k)+2\theta_{2}^{wall}\left\{\chi_{u}\right\}\), hence, applying this formula for \(k=K_{2}\) and for \(k=K_{2}+1 \) we obtain:\begin{equation}
 F(\theta_{2})=\begin{cases}
\theta_2+\Theta_{2}
 & \theta_2\in J_{R}\  \\
\theta_2 +\Theta_{2}+2\theta_{ 2}^{wall}\left\{\chi_{2}\right\}+2\pi K_{2} & \theta_2\in J_{K_{2}} \\
\theta_2+\Theta_{2}+2\theta_{ 2}^{wall}
(-1+\left\{\chi_{2}\right\} )+2\pi (K_{2}+1) & \theta_2\in J_{K_{2}+1}
\end{cases} \quad \label{eq:explicitiemap}
\end{equation}\\ where the   intervals \((J_{R},J_{K_{2}},J_{K_{2}+1})\),  correspond, respectively, to phases with trajectory segments which hit  exactly once only the right side of the step (\(J_{R}\)), those which hit only the upper side of the step exactly \(K_{2}\) times (\(J_{K_{2}}\)) and those hitting only the upper side exactly \(K_{2}+1\) times (\(J_{K_{2}+1}\)), where \(K=K_{2}(e_{1},h) \) (see eq. (\ref{eq:psi}-\ref{eq:kofe1})). Notice that \(\chi_{2}\) is finite since \(\tilde T_2>0 \) for level sets in the step region (yet,    \(\chi_{2}\)  diverges at the step-region boundary when \(q_{2}^{wall}>0\), see  Table \ref{tab:chiutheta}).\\  \\ The order of these intervals on the circle is  \((J_{R},J_{K_{2}},J_{K_{2}+1})\); this follows from the geometry of the swiss-crossed surface or from  realizing that the right (resp. left) most end point of \(J_{R}\) corresponds  to a trajectory which hits the corner with positive (resp. negative) vertical velocity, hence, a small shift into the interval \(J_{U}\) will result in missing the step on the right side and hitting the upper part of the step on the left side, see Figure \ref{fig:orderintervals}.

Under the return map to \(\Sigma_{1}\) the two intervals \(J_{K_{2}},J_{K_{2}+1}\) switch their position and \(J_{U}\) and \(J_{R}\) rotate by \(\Theta_{2}\); This follows from formulae (\ref{eq:explicitiemap}). Indeed,
the dividing trajectory between these two intervals is the trajectory that hits the corner from the direction above the step (i.e. with \(p_{1}>0,p_2<0\)),  and this dividing trajectory is glued to the lower boundary of \( J'_{R}\)
 - since the return map is piece-wise orientation preserving this implies that  \(J_{K_{2}},J_{K_{2}+1}\) must switch their positions - see Figure \ref{fig:orderintervals}. In summary, we proved Eq. (\ref{eq:intervalsorder}).\\
 The lengths of the intervals,  \(\lambda_{\alpha}\) of Eq. (\ref{eq:intervallengths}), follow either from the swiss-cross geometry, or, equivalently, from formulae  (\ref{eq:explicitiemap}), or by considering the phases of the trajectories which hit the step corner  (see  Fig. \ref{fig:orderintervals}):
\begin{align}
\lambda _{J_{K_{2}+1}}&=2\pi\frac{\tilde T_{2}}{T_{2}}
\left\{\frac{T_{1}-\tilde T_{1}}{{\tilde T_{2}}}\right\}
=2\theta_{2}^{wall}\left\{\chi_{2}\right\} \nonumber
\\
\\
\lambda _{J_{K_{2}}}&=2\pi\frac{\tilde T_{2}}{T_{2}}
 \left(  1-\left\{\frac{T_{1}-\tilde T_{1}}{{\tilde T_{2}}}\right\}  \right)=2\theta_{2}^{wall}(1-\left\{\chi_{2}\right\})
\nonumber
\end{align}
 The form of the IEM on the circle is now fully determined by the rotation (\ref{eq:Theta}), the permutation (\ref{eq:intervalsorder}),  and the lengths of the intervals (\ref{eq:intervallengths}).
\\ To determine the 5-IEM on the fundamental interval \([-\pi,\pi)\), we need to identify how the circle-intervals and their images \(J_{\alpha},J_{\alpha}'\), are cut by the chosen fundamental interval, here \([-\pi,\pi)\).  Let  \(\theta_{J_{\alpha}}^{L,R},\theta_{J'_{\alpha}}^{L,R}\in[-\pi,\pi)  \) denote the left and right end points of the circle interval \(J_{\alpha},J_{\alpha}'\)  mod \(2\pi\). Namely, when  \(\theta_{J_{\alpha}}^{L}<\theta_{J_{\alpha}}^{R}\) the circle interval \(J_{\alpha}\)   is not cut by the fundamental interval, so   \(J_{\alpha}^{*}=[\theta_{J_{\alpha}}^L,\theta_{J_{\alpha}}^{R})\subset[-\pi,\pi)  \) whereas   \(\theta_{J_{\alpha}}^{L}>\theta_{J_{\alpha}}^{R}\)   means that  \(J_{\alpha}\) is split to  two intervals, so:       \(J^{*}_{\alpha}=J_{\alpha}^{1}\cup J_{\alpha}^{2}=[-\pi,\theta_{J_{\alpha}}^{R})\cup[\theta_{J_{\alpha}}^{L},\pi)   \), and the same convention is applied to the intervals images. To obtain the 5-IEM,  given an \(\alpha\) such that   \(\theta_{J_{\alpha}}^{L}>\theta_{J_{\alpha}}^{R}\)   we split that interval to two at the phase \(\pi\). Similarly, given an  \(\alpha\) such that    \(\theta_{J'_{\alpha}}^{L}>\theta_{J'_{\alpha}}^{R}\)     we split its pre-image, \(J_{\alpha}\) at \(\theta^{*}\) - the pre-image of \(\pi. \) In the non-degenerate case (i.e. when \(\theta^{L,R}_{J_\alpha},\theta^{L,R}_{J'_\alpha}\neq-\pi, J_\alpha\in\{J_{R},J_{K_2},J_{K_2+1}\}\)),  exactly one of the intervals and exactly one image of an interval is split, so, if  additionally \(\left\{\chi_{2}\right\}\neq0\),  we obtain a 5-IEM. We identify below  the \(J_{R}\) interval end points  and their images and demonstrate that this  completely determines the 5-IEM on \([-\pi,\pi)\).

The left boundary of \(J_{R}\),  \(\theta^{L}_{J_{R}}\),  is the phase of the trajectory which reaches the corner from the right with negative vertical velocity, i.e. it is the phase on \MakeUppercase{\(\Sigma_{1}\)} which arrives to the corner  \((\theta_{1}^{wall},\theta_{2}^{wall}) \) in the swiss cross (see  Figure \ref{fig:orderintervals}). Since the time of passage from \(\Sigma_{1}\) to \(\Sigma^{-}_{1}\) is half of \(\tilde T_{1}\), and since the phases in \(J_{R}\) are rotated by the phase \(\Theta_{2}\), we immediately obtain that :
\begin{equation}
\theta^{L}_{J_{R}}=\theta^{R}_{J_{K_{2}+1}}=\theta_{2}^{wall}-\frac{1}{2}\Theta_{2} \text{  (mod  $2\pi$)},\quad\theta^{L}_{J'_{R}}=\theta^{R}_{J'_{K_{2}}}=\theta_{2}^{wall}+\frac{1}{2}\Theta_{2} \text{  (mod  $2\pi$).} \end{equation}This information, together with the order of the intervals (\ref{eq:intervalsorder}) and their lengths  (\ref{eq:intervallengths}) completely determines the 5 IEM. Indeed,
\begin{equation}
\theta^{R}_{J_{R}}=\theta^{L}_{J_{K_{2}}}=\theta^{L}_{J_{R}}+\lambda_{J_{R}} \text{  (mod  $2\pi$)},
\end{equation}
 hence
\begin{equation}
\quad\theta^{L}_{J'_{K_{2}+1}}=\theta^{R}_{J'_{R}}=\theta^{R}_{J_{R}}+\Theta_{2} \text{  (mod  $2\pi$),}
\end{equation}
 and,
\begin{equation}
\theta^{R}_{J_{K_{2}}}=\theta^{L}_{J_{K_{2}+1}}=\theta^{L}_{J_{K_{2}}}+\lambda_{J_{K_{2}}}\text{ (mod  $2\pi$),}
\end{equation}so,
\begin{equation}
\quad\theta^{L}_{J'_{K_{2}}}=\theta^{R}_{J'_{K_{2}+1}}=\theta^{L}_{J'_{K_{2}+1}}+\lambda_{J_{K_{2}+1}}\text{ (mod  $2\pi$),}
\end{equation}
  and all the intervals' and their images' end points are thus determined by \(\chi_{2},\Theta_{2},\theta_{2}^{wall}\) (all  depending  on \((e_{1},h)\) and on the parameters e.g. \(q_{1,2}^{wall}\)). In particular,  the conditions under which one or more of the 5-intervals in \([-\pi,\pi)\) has zero length can be explicitly formulated: \begin{equation}\label{eq:degeneratelevelsets}
\Theta_{2}(e_{1},h,q_{1}^{wall})=\begin{cases}\Theta^{smooth}(e_{1},h)-2K\theta_{2}^{wall}(h-e_{1},q_2^{wall}) & \text{then }\lambda_{J_{K+1}}=0 \\
\pm2\theta_{2}^{wall}(h-e_{1},q_2^{wall})+2\pi(1+2M) & \text{then }-\pi\in\{\theta^{L,R}_{J_{R}},\theta^{L,R}_{J'_{R}}\} \\
2\theta_{2}^{wall}(h-e_{1},q_2^{wall})(1-2\left\{\chi_{u}\right\})+2\pi(1+2M) & \text{then }-\pi\in\{\theta^{R}_{J_{K}},\theta^{R}_{J'_{K+1}}\}
\end{cases}
\end{equation}where \(K,M \in\mathbb{Z} \). To complete the proof, we need to show that these conditions may be satisfied at most at isolated \(e_{1}\) values. To this aim, we first notice:\begin{lem}\label{lem:independence}For level sets in the step region \(\mathcal{R}^{c}\), the functions  \(\chi_{2},\Theta_{2},\theta_{2}^{wall}\) of \(e_{1}\) are pair-wise independent, and, when \(\Theta_{2}^{smooth}\) is non-constant, they are also pair-wise independent of \(\Theta_{2}^{smooth}\).
\end{lem}
\begin{proof}
 \begin{table}[ht]
  \begin{tabular}{|c|c|c|}\hline
Corner position\ & \(\chi_{2}(h_{1}^{step}),\chi_{2}(h-h_{2}^{step})\) & \(\Theta_{2}(h_{1}^{step}),\Theta_{2}(h-h_{2}^{step}) \)  \\\hline\hline
\(q_{1,2}^{wall}<0\) & \(0,\frac{A(h-h_{2}^{step})}{2\pi}\) &\(\Theta_{2}^{smooth}(h_{1}^{step},h),\frac{\Theta^{smooth}_{2}(h-h_{2}^{step})\theta_{ 1}^{wall}(h-h_{2}^{step})}{\pi}\)  \\\hline
 \(q_{1}^{wall}<0\), \(q_{2}^{wall}>0\) & \(0,\infty\)&
 \(\Theta_{2}^{smooth}(h_{1}^{step},h),\frac{\Theta^{smooth}_{2}(h-h_{2}^{step})\theta_{ 1}^{wall}(h-h_{2}^{step})}{\pi}\)  \\\hline
 \(q_{1}^{wall}>0\), \(q_{2}^{wall}<0\) & \( \frac{\Theta_2^{smooth}(h_{1}^{step},h)}{2\theta_{2}^{wall}(h-h_{1}^{step})} , \frac{A(h-h_{2}^{step})}{2\pi}\)&\(0,\frac{\Theta^{smooth}_{2}(h-h_{2}^{step})\theta_{ 1}^{wall}(h-h_{2}^{step})}{\pi}\) \\\hline
\(q_{1,2}^{wall}>0\) & \(\frac{\Theta_{2}^{smooth}(h_{1}^{step},h)}{2\theta_{2}^{wall}(h-h_{1}^{step})},\infty\)&\(0,\frac{\Theta^{smooth}_{2}(h-h_{2}^{step})\theta_{ 1}^{wall}(h-h_{2}^{step})}{\pi}\)  \\\hline
\end{tabular}
\caption{The  values of \(\chi_{2},\Theta_2\) at the two edges of \(\mathcal{\bar R}^{c}(h)\), where we use the shorthand notation    \(A(h-h_{2}^{step})=\Theta_2^{smooth}(h-h_{2}^{step})(1-\frac{\theta_{1}^{wall}(h-h_{2}^{step})}{\pi} )\). }\label{tab:chiutheta}
\end{table}

 The independence follows from the observation that the functions are smooth non-constant functions (see Tables \ref{tab:thetat12},\ref{tab:chiutheta}) that depend non-trivially on \(e_{1}\) through distinct parameters. For example, \(\Theta_{2}=\Theta_{2}(e_{1},h,q_{1}^{wall})\) whereas \(\theta_{2}^{wall}=\theta_{2}^{wall}(h-e_{1},q_2^{wall}) \) and the dependence of these two functions on \(e_{1}\) through  \(q_{1,2}^{wall}\)  is non-trivial  (i.e., it follows from Eq. (\ref{eq:thetaiwalldef},\ref{eq:Theta}) that \(\frac{\partial^{2}\theta_{2}^{wall}}{\partial q_{2}^{wall}\partial e_{1} }=-\frac{d}{de_{1}} \frac{\omega_2(h-e_1)}{\sqrt{2(h-e_{1}-V_2(q_{2}^{wall}))}}  \) and  \ \(\frac{\partial^{2}\Theta_{2}}{\partial q_{1}^{wall}\partial e_{1} }=-\frac{d}{de_{1}} \frac{2\omega_2(h-e_1)}{\sqrt{2(e_{1}-V_1(q_{1}^{wall}))}}\) hence, with, possibly, the exception of isolated \(e_{1}\) values, these derivatives do not vanish for all level sets in \(\mathcal{R}^{c}\)\() \). Hence, if  they were functionally dependent, i.e.  there was a  \(G(\theta_{2}^{wall},\Theta_{2} ;h,q_{1,2}^{wall})\equiv0\) then \(\frac{d}{dq_{2}^{wall}}G(\theta_{2}^{wall},\Theta _{2};h,q_{1,2}^{wall})=\frac{\partial G }{\partial\theta_{2}^{wall}}\frac{\partial\theta_{2}^{wall}}{\partial q_{2}^{wall} }+\frac{\partial G}{\partial q_{2}^{wall} }\equiv0\) and hence \(\frac{d^{2}}{de_{1}dq_{2}^{wall}}G(\theta_{2}^{wall},\Theta_{2} ;h,q_{1,2}^{wall})=\frac{\partial G }{\partial\theta_{2}^{wall}}\frac{\partial^{2}\theta_{2}^{wall}}{\partial e_{1}\partial q_{2}^{wall} }+\frac{\partial \frac{dG}{de_{1}} }{\partial\theta_{2}^{wall}}\frac{\partial\theta_{2}^{wall}}{\partial q_{2}^{wall} }+\frac{\partial \frac{dG}{de_{1}}}{\partial q_{2}^{wall} }=\frac{\partial G }{\partial\theta_{2}^{wall}}\frac{\partial^{2}\theta_{2}^{wall}}{\partial e_{1}\partial q_{2}^{wall} }=0.\)   Since  \(\frac{\partial^{2}\theta_{2}^{wall}}{\partial e_{1}\partial q_{2}^{wall} }\neq0 \) we conclude that \(\frac{\partial G }{\partial\theta_{2}^{wall}}=0\), namely, there is no such \(G\) with non-trivial dependence on both \(\theta_{2}^{wall}\) and \(\Theta_{2})\).  Similarly, since \(\chi_{2}=\chi_{2}(e_{1},h,q_{1}^{wall},q_{2}^{wall})\), and similarly to the above calculations, the dependence of \(\chi_{2}\)  on both \(q_{1}^{wall}\) and    \(q_{2}^{wall}\) is non-trivial in \(e_{1}\),   the pairs \((\chi_{2},\Theta_{2})\) and \((\chi_{2},\theta_{2}^{wall})\) are functionally independent. Finally, since  \(\Theta_{2}^{smooth}=\Theta_{2}^{smooth}(e_{1},h)\), by the same argument as above, provided \(\frac{\partial\Theta^{smooth}(e_{1},h)}{\partial e_{_{1}}}\neq0\) apart of isolated points, it is pair-wise independent from each of the functions   \(\chi_{2},\Theta_{2},\theta_{2}^{wall}\).  \end{proof}

Now, we can show that (\ref{eq:degeneratelevelsets}) is satisfied at most at isolated \(e_{1}\) values.  For the first two possibilities, both sides of the equation are smooth functions of \(e_{1}\) with the right hand side depending non-trivially on \(q_{1}^{wall}\) whereas the left hand side depending non-trivially on  \(q_{2}^{wall}\). Hence, by the same arguments as in lemma \ref{lem:independence}, the right and left hand side are functionally independent and their difference vanish  at most at isolated \(e_1\) values. \\ For the last row, notice that
\begin{equation}\label{eq:xithetthet2}
\chi_{2}(e_{1},h,q_{1}^{wall},q_{2}^{wall})=\frac{T_{ 2}(h-e_{1})}{\tilde T_{ 2}(h-e_{1})}\frac{ T_{1}(e_1)-\tilde T_{1}(e_1)}{ T_{ 2}(h-e_{1})}=\frac{\Theta_{2}^{smooth}-\Theta_{2}}{ 2\theta_{2}^{wall}},
\end{equation}namely,   \(2\left\{\chi_{u}\right\}\theta_{2}^{wall}=\Theta_{2}^{smooth}-\Theta_{2}-2K_{2}\theta_{2}^{wall} \), hence  the last equation becomes \begin{equation}
2\theta_{2}^{wall}(h-e_{1};q_2^{wall})(1+2K_{2})=2\Theta_2^{smooth}(e_{1},h)-\Theta_{2}(e_{1},h,q_{1}^{wall})-2\pi(1+2M),
\end{equation}
which shows, as above, that it is also satisfied at most at isolated \(e_1\) values. \\ Notice that \(\theta_{2}^{wall}>0\) for all level sets in \(\mathcal{R}^{c}(h)\), so the circle map has always at least two non-trivial components (\(J_{R} \) and \(J_U\)), and in fact, with the exception of isolated points within  \(\mathcal{R}^{c}(h)\), it has three non-trivial components since \(\left\{\chi_{2}\right\}\) vanishes at most at isolated \(e_1\) values.
\begin{figure}
\begin{centering}
\includegraphics[scale=0.35]{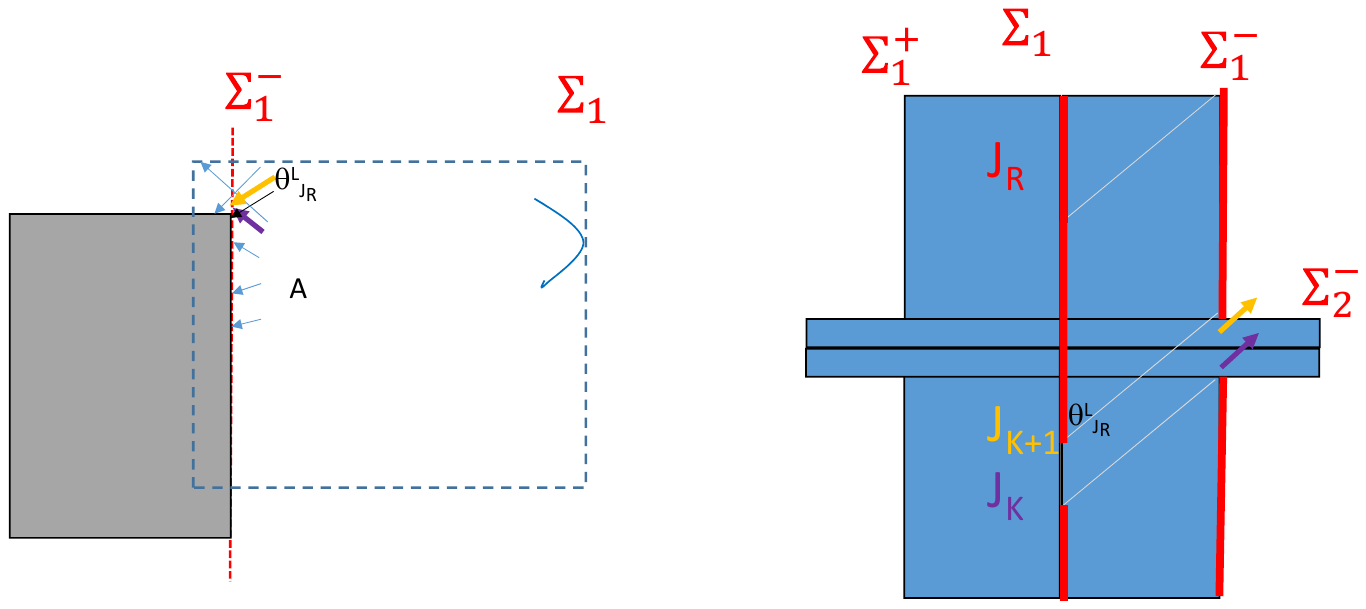}
\includegraphics[scale=0.35]{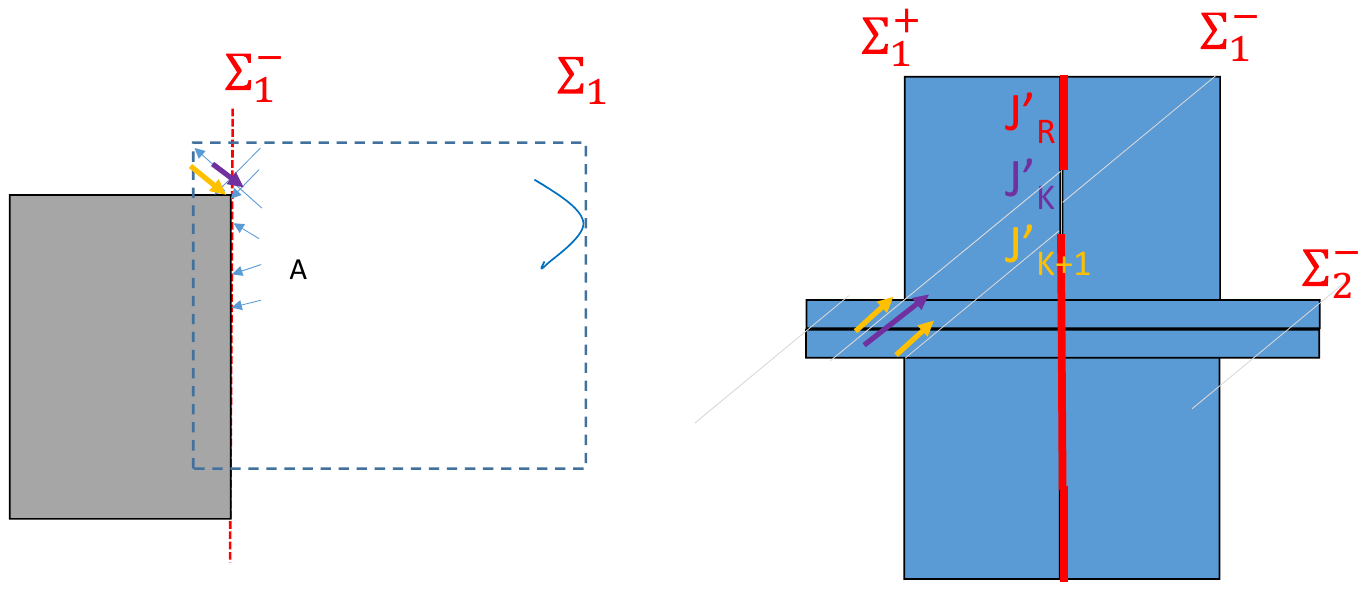}
\end{centering}
\protect\caption{Order of intervals: the intervals \(J_{K_{2}}, J_{K_{2}+1}\) correspond to orbits that bounce, respectively, \(K_{2}, K_2+1\) times above the step. On the SW surface these are the orbits that enter the horizontal arm and wrap, respectively, \(K_{2}, K_2+1\) times around it before returning to the vertical arm. Thus, their order on \(\Sigma_{1}\) is reversed by the flow (Eq. (\ref{fig:orderintervals})). The phase \(\theta_{J_{R}}^L\) denotes the left edge of \(J_R\) - the phase that separates the orbits that enter the horizontal arm from those that bounce off the vertical arm boundary, namely the vertical boundary of the step.  }\label{fig:orderintervals}
\end{figure}

 \end{proof}

   \section{Additional properties of the family of IEM\label{sec:moreproperties}.}

Theorem \ref{thm:mainmap} implies that the dynamics  for level sets in the step set are completely determined by the numerical properties of \(\chi_{2},\Theta_{2},\theta_{2}^{wall}\).    All these functions  depend smoothly on \(e_{1}\) for  the level sets in  \(\mathcal{R}^{c}(h)\) and are non-constant functions - indeed, their  values at the boundaries of  \(\mathcal{R}^{c}(h)\) are always distinct - see Tables \ref{tab:thetat12} and  \ref{tab:chiutheta}.
 Hence, these functions attain both rational and irrational values as \(e_{1}\) is varied (in some cases, but not all, these functions are also monotone in \(e_{1}\)). While one may suspect that this implies that for almost all \(e_{1}\) values the dynamics are uniquely ergodic, it is difficult to check directly when the corresponding IEM satisfies the Veech condition (see \cite{Fraczek2018}). Indeed, while Lemma \ref{lem:independence} states that the functions \(\chi_{2},\Theta_{2},\theta_{2}^{wall}\) are pair-wise independent, and, in Theorem \ref{thm:mainmap} we established that the lengths of the intervals of the IEM are non-zero with the exception of isolated \(e_{1}\) values, more delicate relations between the intervals lengths may arise. Indeed, rewriting Eq. (\ref{eq:xithetthet2}) as:
\begin{equation}\label{eq:thetaxitheta2}
\Theta_{2}^{smooth}(e_{1},h)=2\theta_{2}^{wall}(h-e_{1},q_2^{wall})\chi_{2}(e_{1},h,q_{1}^{wall},q_{2}^{wall})+\Theta_{2}(e_{1},h,q_{1}^{wall}),
\end{equation}shows that in the linear case, where \(\Theta_{2}^{smooth,LO}=2\pi\frac{\omega_{2}}{\omega_{1}}\), the three functions are functionally related! The implications of this dependence on the dynamics and the properties of it for general nonlinear oscillators are yet to be explored. For now we show, by analyzing the properties of these functions, that, for some cases minimal dynamics arise and in others non-minimal dynamics arise.
\\ \\  In particular, we establish that there can be isolated strongly resonant level sets at which orbits of different periods co-exist (e.g. if \(  {\frac{\Theta_{2}}{2\pi },\frac{2\theta_{2}^{wall}}{2\pi },\left\{\chi_{2}\right\}}\in\mathbb{\mathbb{Q})},\) level sets for which periodic and quasi-periodic motion co-exist (e.g. when \(  \{\frac{\Theta}{2\pi },\frac{2\theta_{2}}{2\pi }\}\in\mathbb{\mathbb{Q}},\left\{\chi_{2}\right\}\notin\mathbb{\mathbb{Q}}\) such a case may emerge) and  isolated level sets in   \(\mathcal{R}^{c}(h)\) at which the IEM reduces to a rotation (when \(\left\{\chi_{2}\right\}=0\) so the directional flow on SW has a diagonal trajectory in the horizontal arm).
  Notice that the analogous computations for the return map to \(\Sigma_{2}\) amounts to replacing \(1\leftrightarrow2\)  in all the above definitions.

 In particular, we notice the special role the function \( \chi_{2}\) plays: its magnitude  controls the number of bounces experienced by phases in \(J_{K_{2}}\) (recall that \(K_{2}=\left\lfloor\chi_{2}\right\rfloor\)) and its phase, \(\{\chi_{2}\}\), controls the division of \(J_{U}\) to two intervals (recall that  \(\lambda _{J_{K_{2}+1}}=2\theta_{2}^{wall}\{\chi_{2}\}\)). Hence, we study the dependence of  \( \chi_{2}\) on \(e_{1}\) and on the parameters \(q_{1,2}^{wall}\). We begin with two simple cases where we can completely characterize the dynamics: \begin{cor}\label{cor:xuzero}
 For level sets  in \(\mathcal{R}^{c}(h)\) for which  \(  \{\chi_{2}\}=0\)  the return map to \(\Sigma_{1}\)  is of only 2 intervals, namely it corresponds to a rotation by \(\Theta_{2}\), and is thus ergodic iff \(\Theta_{2}/2\pi\notin\mathbb{Q}\).      \end{cor}

     \begin{proof}  By (\ref{eq:intervallengths}),   \(  \{\chi_{2}\}=0\)  implies that   \(\lambda _{J_{K_{2}}}=2\theta_{2}^{wall}>0\) and \(\lambda_{J_{K_{2}+1}}=0\), hence,  (\ref{eq:explicitiemap}) becomes a rotation by \(\Theta_{2}\).  \end{proof}

 In terms of the directed motion on the L-shaped billiard, the condition    \(  \{\chi_{2}\}=0\)  corresponds to the case  of a diagonal orbit connecting the corners of  the horizontal arm. If, additionally,  \(\Theta_{2}/2\pi\in\mathbb{Q}\) then this orbit is also a diagonal of the vertical arm.  Notably,  if \(q_{2}^{wall}>0\), close to the boundary of  \(\mathcal{R}^{c}(h)\) the horizontal sleeve becomes  narrow (see Fig. \ref{fig:theta2walldelta}) and thus there are many level sets at which   \(  \{\chi_{2}\}=0\):
\begin{lem}
\label{lem:q2posxu0} If \(q_{2}^{wall}>0\), for all \(h>h^{step}\), there are countable infinite level sets  in \(\mathcal{R}^{c}(h)\) for which  \(  \{\chi_{2}\}=0\).       \end{lem}
     \begin{proof} Since, for \(q_{2}^{wall}>0\),    \(\tilde T_{2}(h-e_{1};q_{2}^{wall})\rightarrow 0\) as \(  e_1 \rightarrow h-h_2^{step}\) whereas \(T_{1}(e_{1})-\tilde T_{1}(e_1;q_1^{wall})\)  attains a finite positive limit (since \(h>h^{step}\)), the smooth function    \(\chi_{2}(e_1,h)=\frac{T_{1}(e_{1})-\tilde T_{1}(e_1)}{{\tilde T_{2}(h-e_1)}}\) in the open interval  \(\mathcal{I}^{c}(h)\)  becomes infinite on the interval right boundary, hence it passes
through integer \ values at countable infinite\ values of \(e_{1}\).    \end{proof}
Another case which allows a complete characterization of the motion is when  \(\Theta_{2}\) is rational and \(\theta_{2}^{wall}\) is small:
\begin{lem}\label{cor:xurational}
For level sets in  \(\mathcal{R}^{c}(h)\) for which \(\Theta_{2}=\frac{2\pi m}{n} \) and \(2\theta_{2}^{wall}<\frac{2\pi }{n}\), the IEM to \(\Sigma_{1}\) is non-ergodic. For such level sets, if    \(\{\chi_{u}\}\notin\mathbb{Q} \) the motion is dense on a union of open intervals and is periodic on its complement.  If    \(\{\chi_{u}\}\in\mathbb{Q} \) , all i.c. are periodic, yet, there are two distinct periods.   All the above  conditions are realizable for some level sets and wall positions.\end{lem}
\begin{proof}
 Let \(I=[-\pi ,\pi )\backslash \bigcup^{n-1} _{j=0}F^{j}(J_{U})\subset J_{R}\) where \(J_{U} =J_{K_{2}}\cup J_{K_{2}+1}\). Since here  \(\lambda_{J_U}=2\theta_{2}^{wall}<\frac{2\pi }{n}\), this is a non-empty set. It is invariant since the end points of \(J_{U}\)  belong to \(J_{R}\),  so the end points are \(n\)-periodic. Hence, all the i.c. in \(I\) are \(n\) periodic and thus \(F\) is non-ergodic on the circle.

\begin{figure}
\begin{centering}
\qquad \qquad \qquad \qquad \qquad \qquad \includegraphics[scale=0.3]{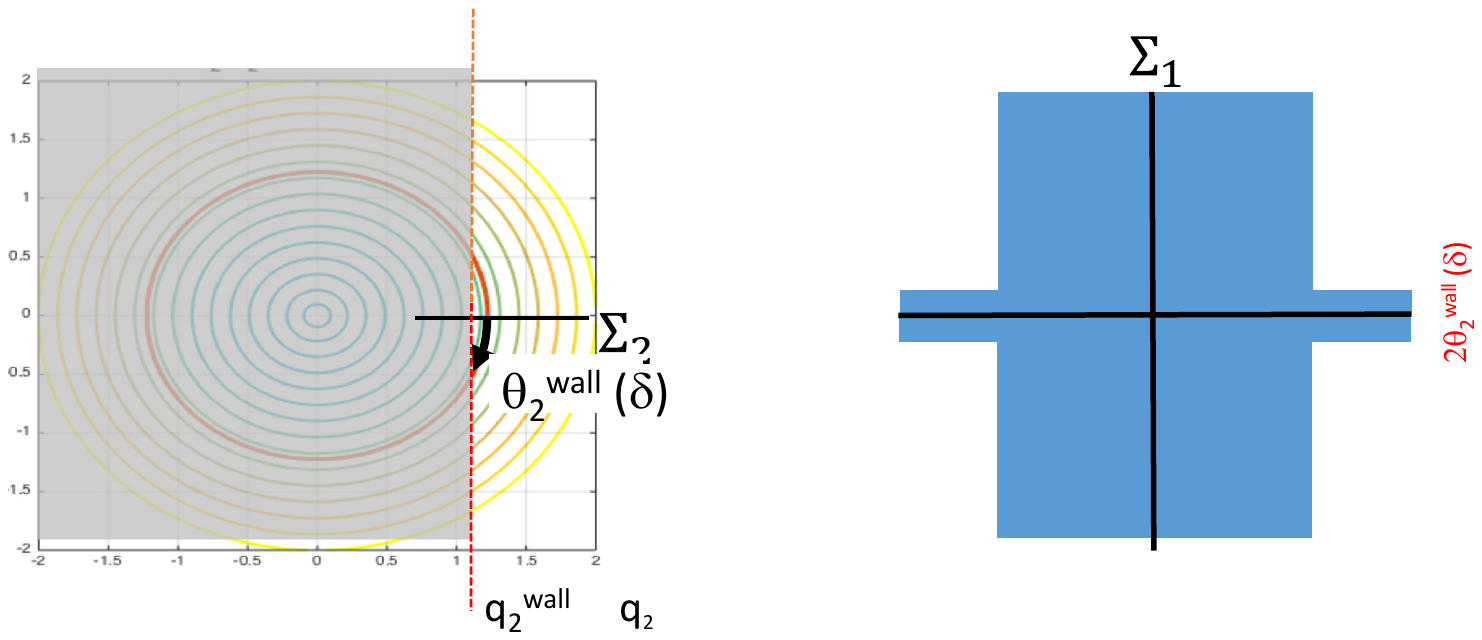}
\end{centering}
\protect\caption{The dynamics for small \(\theta_{2}^{wall}\).  }\label{fig:theta2walldelta}
\end{figure}

 The dynamics in the complement to \(I\), namely the invariant set \( \bigcup^{n-1} _{j=0}F^{j}(J_{U})\), depends on the numerical value of \(\{\chi_{2}\}\). Notice that for all i.c. in \(J_{K_{2}}\), \(F^n(\theta_2)=\theta_2+2\theta_{ 2}^{wall}\left\{\chi_{u}\right\}\in J_{U}\) whereas for all i.c. in \(J_{K_{2}+1}\), \(F^n(\theta_2)=\theta_2-2\theta_{ 2}^{wall}(1-\left\{\chi_{u}\right\} )\in J_{U}\), namely, \(F^n(\theta_2)\) is a 2-IEM on \(J_U\), hence it is periodic for    \(\{\chi_{u}\}=\frac{p}{q}\in\mathbb{Q} \) and is dense in \(J_{U}\) otherwise. In the periodic case, initial conditions in \(I \) are \(n\)-periodic whereas initial conditions in its complement are \(nq\) periodic.  Finally, since the functions \(\Theta_{2},\chi_{2},\theta_{ 2}^{wall}\) are continuous (in fact, smooth) non-constant functions of \(e_{1}\) in   \(\mathcal{I}^{c}(h)\)  and since \(\Theta_{2}=\Theta_{2}(e_{1},h,q_1^{wall})\)  we obtain that for every \(h,q_1^{wall}\) there is a countable set of \(e_{1}\) values in   \(\mathcal{I}^{c}(h)\), \(e_{1}^*=e_{1}(\frac{m}{n},h,q_1^{wall})\) for which  \(\Theta_{2}=\frac{2\pi m}{n} \). Notice that \(\Theta_{2}\) does not depend explicitly  on      \(q_{2}^{wall}\).  Fixing \(\frac{m}{n},h,q_1^{wall}\),  there is a      \(q_{2}^{wall}(\delta)\) value such that \(2\theta_{2}^{wall}(h-e_{1}^*,q_{2}^{wall})<\frac{2\pi }{n}\) for all     \(q_{2}^{wall}>q_{2}^{wall}(\delta)\);  indeed, choose     \(q_{2}^{wall}(\delta)>0\) such that \(V_{2}(q_2^{wall}(\delta))=h-e_{1}^{*}-\delta\), so, by (\ref{eq:thetaiwalldef}) and (\ref{eq:thetailimits}),  for small \(\delta\), \(\theta_{2}^{wall}(\delta)=\theta_{2}^{wall}(h-e_{1}^*,q_{2}^{wall}(\delta)) \) is monotone decreasing (in \(\delta\)) to \(0\) (see Fig. \ref{fig:theta2walldelta}). In particular,  there exists \(\delta^{*}(n)\) such that for all  \(\delta\in(c\delta^{*}(n),\delta^{*}(n))\),  for any \(0<c<1\), the impact angle satisfies \(2\theta_{2}^{wall}(c\delta^{*}(n))<2\theta_{2}^{wall}(\delta)<\frac{2\pi }{n}\), so it is  small yet bounded away from \(0,\)   as needed for the smooth dependence on \(e_1\) near \(e_{1}^{*}\). In particular, for this range of  \(q_{2}^{wall}(\delta)\) values,  motion on the level set \((e_{1}^*,h-e_1^*)\) is \(n\)-periodic for the set \(I\) as described above. Moreover, on this level set, from (\ref{eq:psi}),   \(\chi_{2}(\delta)=\frac{\Theta^{smooth}(e_1^*,h)-\frac{2\pi m}{n}}{\theta_{2}^{wall}(\delta)}>0\), hence, it is a continuous monotone increasing function of \(\delta,\) and thus,      \(\{\chi_{u}\}\notin\mathbb{Q} \)
 for almost all \(\delta\) values in the interval and there is  a countable set of \(\delta \) values for  which     \(\{\chi_{u}\}\in\mathbb{Q} \). Namely, we established that these conditions are always realizable by varying the parameter \(q_{2}^{wall}\).
%\end{proof}
\end{proof}
Notice that for sufficiently small \(c\) in the above proof the function \(\chi_{2}(\delta)\) becomes large, as in lemma \ref{lem:q2posxu0}, therefore, \(\{\chi_{2}(\delta)\}\)  vanishes at some isolated \(\delta\) values.
Finally, since \(\chi_{2}\) is continuous, its range for \(e_{1}\in\mathcal{I}^{c}(h)\) is at least as large as the interval \((\chi_{2}(h^{step}_1),\chi_{2}(h-h^{step}_2)) \). When one of these values is an integer, the behavior below and above this energy changes. Thus, Table \ref{tab:chiutheta} provides conditions for energy values at which bifurcations occur.
For linear oscillators, we can find the ranges explicitly, see section \ref{sec:linearosc}.

All the above properties were stated for the return map to \(\Sigma_{1}\), creating an artificial asymmetry between the horizontal and vertical directions. The same results apply to the return map to the \(\Sigma_2\) section by reversing the roles of 1 and 2  and horizontal and vertical in all definitions.

\section{The step dynamics for linear oscillators\label{sec:linearosc} }

For the quadratic potentials (Eq. (\ref{eq:LOpot})), the L-shaped billiard tables vertices are found explicitly and the direction of motion is fixed. We begin this section by proving Theorems
\ref{thm:mainflowlin} regarding the linear-oscillators-step dynamics and continue with additional observations regarding the singular level sets for this case.\
  \begin{proof} \textit{of Theorem
\ref{thm:mainflowlin}} : The transformation to action angle coordinates for linear oscillators, with the convention (\ref{eq:defsigmai}), becomes: \((q_{i}(t),p_{i}(t))=(\sqrt{\frac{2I_{i}}{\omega_{i}}}\cos \theta _{i}(t),-\sqrt{2I_{i}\omega_{i}} \sin \theta _{i}(t)), H_{i}=\omega_i I_{i}\), and \(I_{i}=\frac{1}{2}(\frac{p_i^2}{\omega_i}+\omega_iq_i^2) \).  Hence, for \(e_{i}>h_i^{step}=\half \omega^{2}_{i}(q_i^{wall})^2\):\begin{equation}\label{eq:thetaiwallLO}
\theta_{i}^{wall,LO}(e_{i};q_i^{wall})=\arccos \sqrt{\frac{\omega_{i}}{2I_{i}}}q_{i}^{wall}=\arccos \frac{\omega_{i}q_{i}^{wall}}{\sqrt{2e_{i}}}\in\begin{cases}(\frac{\pi}{2},\pi)  & q_{i}^{wall}<0 \\
(0,\frac{\pi}{2}) & q_{i}^{wall}>0. \\
\end{cases}
\end{equation}  Since  the frequencies of linear oscillators are independent of the energy, the  direction of motion  in the iso-energetic billiard family,  \(\mathbf{\mathcal{B}}(h)\), is fixed to \((\omega_1,\omega_2)\) for all the level sets. Hence, for a given energy surface \(h=e_{1}+e_2>h^{step}\), the linear oscillator step dynamics on each level set \((e_{1},e_2=h-e_1)\in\mathcal{R}^{c}(h)\) are conjugated  to the directed billiard motion, in direction \((\omega_{1},\omega_2)\), in the L-shaped billiards
\(L(\pi,\pi,\arccos \frac{\omega_{1}q_{1}^{wall}}{\sqrt{2e_{1}}},\arccos \frac{\omega_{2}q_{2}^{wall}}{\sqrt{2(h-e_{1})}})\). Moreover, since \begin{equation}
\frac{d\theta_{1}^{wall,LO}(e_{1};q_1^{wall})}{de_{1}}=\frac{\omega_{1}q_{1}^{wall}}{2e_{1}}\frac{1}{\sqrt{2e_{1}-\left({\omega_{1}q_{1}^{wall}}\right)^2}}
\end{equation}

and

\begin{equation}
\frac{d\theta_{2}^{wall,LO}(h-e_{1};q_1^{wall})}{de_{1}}=-\frac{\omega_{2}q_{2}^{wall}}{2(h-e_{1})}\frac{1}{\sqrt{2(h-e_{1})-
\left({\omega_{2}q_{2}^{wall}}\right)^2}},
\end{equation}the widths of the arms are monotone in \(e_{1}\) and are of opposite monotonicity iff \(q_{1}^{wall}q_{2}^{wall}>0\).

\end{proof}

As shown in section \ref{sec:moreproperties}, the monotonicity, bounds and limits of the functions \(\Theta_{2},\chi_{2},\theta_{ 2}^{wall}\) determine the variety of behaviors of the dynamics on iso-energy surfaces. For linear oscillators,  \(\Theta_{2}^{LO},\theta_{ 2}^{wall,LO}\)  are  monotone in the step region whereas:    \begin{lem}
 For all \(h>h^{step}\), the function   \(  \chi_{2}^{LO}(e_{1},h)\) is monotone iff \(q_{1}^{wall}q_{2}^{wall}<0\).       \end{lem}
\begin{proof}Observe that for linear oscillators  \(q_{i}^{wall}\tilde T_{i}'^{LO}(e_{i};q_{i}^{wall})>0\) (see Eqs. (\ref{eq:thetaiwalldef},\ref{eq:thetailo}) and recall that \( T_{i}'^{LO}=0\), hence the result follows from the definition (\ref{eq:psi}) of \(\chi_{2}(e_{1},h)\), or, from direct differentiation of   \(\chi_{2}^{LO}(e_{1},h)=\frac{\omega_{2}}{\omega_{1}}\frac{(\pi-\arccos \frac{\omega _{1}q_{1}^{wall}}{\sqrt{2e_{1}}})}{  \arccos \frac{\omega _{2}q_{2}^{wall}}{\sqrt{2(h-e_{1})}}}=
\frac{\omega_{2}}{\omega_{1}}\frac{\pi-\theta_{1}^{wall,LO}(e_{1};q_1^{wall})}{\theta_{2}^{wall,LO}(h-e_{1};q_2^{wall})}\) (see (\ref{eq:thetxiLO})):
\begin{equation}
\chi^{LO'}_{2}(e_{1})=\frac{\omega_{2}\left(-\frac{\omega_{1}q_{1}^{wall}\arccos \frac{\omega _{2}q_{2}^{wall}}{\sqrt{2(h-e_{1})}}}{2e_1\sqrt{2e_{1}-(\omega _{1}q_{1}^{wall})^{2}}}+\frac{\omega _{2}q_{2}^{wall}(\pi-\arccos \frac{\omega _{1}q_{1}^{wall}}{\sqrt{2e_{1}}})}{2(h-e_1)\sqrt{2(h-e_{1})-(\omega _{2}q_{2}^{wall})^{2}}}\right)}{\omega_{1}(\arccos \frac{\omega _{2}q_{2}^{wall}}{\sqrt{2(h-e_{1})}})^2}.
\end{equation}
The denominator is always positive (it approaches \(0\) when    \(q_{2}^{wall}>0\) and  \(e_{1}\nearrow h-\frac{1}{2}(\omega _{2}q_{2}^{wall})^{2}\) ) so the sign is determined by the numerator. If     \(q_{1}^{wall}q_{2}^{wall}<0\) both terms in the numerator have the same sign so  \(\chi_{2}\) is monotone. If    \(q_{1}^{wall}q_{2}^{wall}>0\), the first term in the numerator diverges to \(-\text{sign}(q_{1}^{wall})\infty\) as \(e_{1}\searrow\frac{1}{2}(\omega _{1}q_{1}^{wall})^{2}\),  the second term diverges to \(\text{sign}(q_{2}^{wall})\infty\)    as as \(e_{1}\nearrow h-\frac{1}{2}(\omega _{2}q_{2}^{wall})^{2}\),  hence  \(\chi_{2}'(e_{1})\) changes sign and   \(\chi_{2}\) is non-monotone.
\end{proof}

\begin{cor}
If \(q_{2}^{wall}>0\), for all \(h>h^{step}\), the step region has countable infinite level sets at which \(\left\{\chi^{LO}_{2}\right\}=0\), namely at which the return map to \(\Sigma_{1}\) reduces to a two intervals rotation on the circle. For \(q_{2}^{wall}<0\), for sufficiently large \(h \), the number, \(N^{2}_{osc}(h)\), of such level sets when   \(q_{1}^{wall}<0\) is at least \(\left\lfloor\frac{1}{2}\frac{\omega_{2}}{\omega_{1}}\right\rfloor\)whereas if  \(q_{1}^{wall}>0\) there are \(\left\lfloor\frac{3}{2}\frac{\omega_{2}}{\omega_{1}}\right\rfloor\) such level sets. The same results hold for the return map to \(\Sigma_{2}\) when replacing the roles of \(1\leftrightarrow2\) in the above statements.  \end{cor}

\begin{proof} First, it follows from (\ref{eq:thetxiLO})  that in the step region \(\chi_{2}^{LO}(e_{1},h)\) is a smooth non-oscillatory function which diverges only at the step region upper boundary (and this occurs iff \(q_{2}^{wall}>0\)). Hence, for any fixed \(h \), there is at most countable infinite level sets \(N^{2}_{osc}(h) \) at which   \(\left\{\chi_{2}^{LO}(e_{1})\right\}\) may vanish. The  edge  values - the values of \(\chi_{2}^{LO},\Theta_{2}^{LO}\) at the end points of the step-region (namely calculating (\ref{eq:thetailo}),(\ref{eq:thetxiLO}) at \(e_{1}=h_1^{step}\) and at \(e_1=h-h_2^{step}\)) and their monotonicity property  are listed in Table
\ref{tab:chiuthetaLO}, where  \begin{equation}
\theta_1^*(h)=\theta_{1}^{wall}(h-h_2^{step};q_{1}^{wall})=\arccos \frac{\omega _{1}q_{1}^{wall}}{\sqrt{2h-(\omega _{2}q_{2}^{wall})^2}}
\end{equation}

\begin{equation}\theta_2^*(h)=\theta_{2}^{wall}(h-h_1^{step};q_{2}^{wall})=\arccos \frac{\omega _{2}q_{2}^{wall}}{\sqrt{2h-(\omega _{1}q_{1}^{wall})^2}}.\end{equation}

  \begin{table}[ht]
  \begin{center}
  \begin{tabular}{|l|l|l|}\hline
Corner position\ & \(\chi_{2}^{LO}(h_{1}^{step})\rightarrow\chi^{LO}_{2}(h-h_{2}^{step})\) & \(\Theta^{LO}_{2}(h_{1}^{step})\rightarrow\Theta^{LO}_{2}(h-h_{2}^{step})\)    \\\hline\hline
\(q_{1,2}^{wall}<0\) & \(\qquad0\nearrow\searrow\frac{\omega_{2}\left(1-\frac{1}{\pi}\theta_1^*(h)\right)}{\omega_{1}}
 \) &\(2\pi\frac{\omega_{2}}{\omega_{1}}\searrow2\frac{\omega_{2}}{\omega_{1}}\theta_1^*(h)\)  \\\hline
 \(q_{1}^{wall}<0\), \(q_{2}^{wall}>0\) & \(\qquad0\nearrow\quad\infty\)&
 \(2\pi\frac{\omega_{2}}{\omega_{1}}\searrow 2\frac{\omega_{2}}{\omega_{1}}\theta_1^*(h)\) \\\hline
 \(q_{1}^{wall}>0\), \(q_{2}^{wall}<0\) & \( \frac{\omega_{2}}{\omega_{1}}\frac{\pi}{  \theta_2^*(h)}
 \searrow\frac{\omega_{2}\left(1-\frac{1}{\pi}\theta_1^*(h)\right)}{\omega_{1}}
 \)
 &\(\qquad0\nearrow 2\frac{\omega_{2}}{\omega_{1}}\theta_1^*(h)\)  \\\hline
\(q_{1,2}^{wall}>0\) & \(\frac{\omega_{2}}{\omega_{1}}\frac{\pi}{  \theta_2^*(h)}
 \searrow\nearrow\infty\)&\(\qquad0\nearrow2\frac{\omega_{2}}{\omega_{1}}\theta_1^*(h)\) \\\hline
\end{tabular}
\end{center}
\caption{The  values of \(\chi^{LO}_{2},\Theta^{LO}_{2}\) at the edges of \(\mathcal{R}^{c}(h)\) and their monotonicity properties. }\label{tab:chiuthetaLO}
\end{table}

%%%%%

For any energy  \(h\), these values supply the range of  \(\chi_{2}^{LO}\) in the monotone cases  (second and third rows in the tables) and a lower bound on its range in the non-monotone cases (first and last rows).  The number of iso-energy level sets at which \(\{\chi_{2}^{LO}=0\} \) (at which the directional motion in the horizontal arm of the SW surface is diagonal) is determined by the number of integer values contained in the range of    \(\chi_{2}^{LO}\).  The second and fourth rows of Table \ref{tab:chiuthetaLO} show that the range is infinite when \(q_{2}^{wall}>0\), proving the first statement of the corollary.  Table \ref{tab:xulolimits} shows the asymptotic  edge values at large energies, using the observation that \(\theta_{1,2}^*(h)\rightarrow\frac{\pi }{2}\).
The rest of the corollary follows from this table -  for    \(q_{2}^{wall}<0\), the first row of Tables   \ref{tab:xulolimits} corresponds to the non-monotone case whereas the third row corresponds to the monotone case. Since \(N^{2}_{osc}(h)\) are integers, for sufficiently  large \(h \) the limiting values and \(N^{2}_{osc}(h)\) are identical.
Finally,
 by symmetry,  replacing the roles of \(1\leftrightarrow2\), provides  the estimates for     \(N^{1}_{osc}(h)\),\ the number of oscillations in the vertical arm before returning to the section \(\Sigma_{2}\).

\begin{table} [ht]
  \begin{center}
  \begin{tabular}{|c|c|c|c|c|}\hline
Corner position\ & \(\chi^{LO}_{2}\) & \(\Theta_{2}^{LO}\ \)& \(N_{osc}^2\)& \(N_{osc}^1\)   \\\hline\hline
\(q_{1,2}^{wall}<0\) & \(0\nearrow\searrow \frac{\omega_{2}}{2\omega_{1}}\) &  \(2\pi\frac{\omega_{2}}{\omega_{1}}\searrow\pi\frac{\omega_{2}}{\omega_{1}}\)
 & \(\geqslant\left\lfloor\frac{1}{2}\frac{\omega_{2}}{\omega_{1}}\right\rfloor\) & \(\geqslant\left\lfloor\frac{1}{2}\frac{\omega_{1}}{\omega_{2}}\right\rfloor\) \\\hline
 \(q_{1}^{wall}<0\), \(q_{2}^{wall}>0\) & \(0\nearrow\infty\) &
 \(2\pi\frac{\omega_{2}}{\omega_{1}}\searrow\pi\frac{\omega_{2}}{\omega_{1}}\)
 & \(\infty \)
&\(\left\lfloor\frac{3}{2}\frac{\omega_{1}}{\omega_{2}}\right\rfloor \) \\\hline
 \(q_{1}^{wall}>0\), \(q_{2}^{wall}<0\) &
 \(\frac{2\omega_{2}}{\omega_{1}}\searrow\frac{\omega_{2}}{2\omega_{1}}\)
  &  \(0\nearrow\pi\frac{\omega_{2}}{\omega_{1}}\)
&\(\left\lfloor\frac{3}{2}\frac{\omega_{2}}{\omega_{1}}\right\rfloor \)  & \(\infty \) \\\hline
\(q_{1,2}^{wall}>0\) & \(\frac{2\omega_{2}}{\omega_{1}}\searrow\nearrow\infty\) &\(0\nearrow\pi\frac{\omega_{2}}{\omega_{1}}\)&\(\infty\)  & \(\infty \)\\\hline
\end{tabular}
\end{center}
\caption{The behavior of \(\chi_{2}^{LO},\Theta_{2}^{LO}\)  on   \(\mathcal{R}^{c}(h)\) at large \(h \) and the number of oscillations, \(N_{osc}^i\)(h), of \(\left\{\chi^{LO}_{i}\right\}\) in the family of iso-energy return maps to \(\Sigma_{i}\) for sufficiently large \(h.\) \label{tab:xulolimits}}
\end{table}
\end{proof}

Table \ref{tab:xulolimitshstep}  displays the edge values  at energies near \(h^{step}\) (i.e. for \(h= h^{step}+\eta\), and small \(\eta\)). Notice that for such \(h\) values \(\theta_{i}^*(h)=\sqrt{\eta/h_{i}^{step}}\) if \(q_{i}^{wall}>0\) and \(\theta_{i}^*(h)=\pi-\sqrt{\eta/h_{i}^{step}}\) if \(q_{i}^{wall}<0\). We see that when \(q_{2}^{wall}>0\), infinite number of oscillations occur  for arbitrary small \(\eta\), whereas in the other cases, the number of oscillations scales with   \(\sqrt{\eta}\).

 \begin{table} [ht]
  \begin{center}
  \begin{tabular}{|c|c|c|c|}\hline
Corner position\ & \(\chi^{LO}_{2}\) & \(\Theta_{2}^{LO} \)& \(N^{2}_{osc}(h^{step}+\eta)\)   \\\hline\hline
\(q_{1,2}^{wall}<0\) & \(0\nearrow\searrow a_{1} \frac{\omega_{2}}{\omega_{1}}\) &  \(2\pi\frac{\omega_{2}}{\omega_{1}}\searrow2\pi\frac{\omega_{2}}{\omega_{1}}(1-a_{1})\)
 &\(\gtrsim \left\lfloor a_1\frac{\omega_{2}}{\omega_{1}}\right\rfloor \) \\\hline
 \(q_{1}^{wall}<0\), \(q_{2}^{wall}>0\) & \(0\nearrow\infty\) & \(2\pi\frac{\omega_{2}}{\omega_{1}}\searrow2\pi\frac{\omega_{2}}{\omega_{1}}(1-a_{1})\)& \(\infty\) \\\hline
 \(q_{1}^{wall}>0\), \(q_{2}^{wall}<0\) &
 \(\frac{\omega_{2}}{\omega_{1}}(1+a_{2})\searrow\frac{\omega_{2}}{\omega_{1}}(1-a_{1}) \)
  &  \(0\nearrow\pi a_{1}\)
&\(\left\lfloor\frac{\omega_{2}}{\omega_{1}}a_1 (1+\sqrt{\frac{h_{1}^{step}}{h_{2}^{step}}})\right\rfloor\)  \\\hline
\(q_{1,2}^{wall}>0\) & \(\frac{1}{a_{2}}  \frac{\omega_{2}}{\omega_{1}}
 \searrow\nearrow\infty\) &\(0\nearrow\pi a_{1}\)&\(\infty\) \\\hline
\end{tabular}
\end{center}
\caption{The  values of \(\chi_{2}^{LO},\Theta_{2}^{LO}\) at the edges of   \(\mathcal{R}^{c}(h=h^{step}+\eta)\) for  small \(\eta\), namely \(\chi^{LO}_{2}(h_{1}^{step}),\chi^{LO}_{2}(\eta+ h^{step}-h_{2}^{step})\) and  \(N^{2}_{osc}(h)   \). The values of  \(N^{1}_{osc}(h)   \)  in the first and second rows are found by switching \(1\leftrightarrow2\) (as in Table \ref{tab:xulolimits}). For shorthand notation  we denote here \(a_{i}=\frac{1}{\pi }\sqrt{\frac{\eta}{h_{i}^{step}}},i=1,2.\)}\label{tab:xulolimitshstep}
\end{table}

\section{Summary and discussion}

An integrable mechanical Hamiltonian system with a step barrier in the configuration space which is aligned with the continuous symmetries of the integrable Hamiltonian produces dynamics that are not Lioville integrable, yet are analyzable.  An experimental setup which realizes such a theoretical model has been suggested (Fig \ref{fig:models}). In such models, the motion on energy surfaces is foliated by level sets, yet, the motion on a range of iso-energy  level sets is non-integrable and is conjugated to the motion on a family of genus 2 flat surfaces or, equivalently, to an L-shaped billiard (Theorem \ref{thm:mainflow}). The return map to a Poincar\'e section for this range of level sets is a 5 interval exchange map, and the lengths of the intervals change non-trivially along the iso-energy family of level sets (Theorem \ref{thm:mainmap}).  For the case of Linear oscillators the L-shaped billiard dimensions and thus the intervals lengths are found explicitly (Theorem \ref{thm:mainflowlin}) whereas for general non-linear oscillators they are given up to quadratures. While our main example included a single step,   the same strategy may be applied to any barrier geometry which combines horizontal and vertical barriers. The flow of the HIS (\ref{eq:modelsham}) with such barriers is conjugated, for any given level set, to a directional motion in the angles' space on  nibbled rectangles with rectangular holes as analyzed in \cite{frkaczek2019recurrence}. An important conclusion is that above certain energy the energy surfaces of   (\ref{eq:modelsham}) are foliated by several families of level set surfaces; within any such family the geometry varies smoothly, and different families have distinct topology. Namely, on the same energy surface there are families of  level-set surfaces with different  number of connected components and different numbers of holes (see corollary \ref{cor:nontrivialtop} and Fig. \ref{fig:torus1}).

The implications of our findings are intriguing; First, the statistics of a typical observable of such mechanical systems (i.e. an observable which does not depend only on the energy distribution among the d.o.f.) are now related to  the delicate theories derived for studying IEM and Teichmuler flows on moduli spaces.
Second, by considering soft steep potentials instead of impacts, the topology of the energy surfaces remains as complex as the one constructed here (then the motion is not expected to be foliated to level sets).  Higher dimensional extensions, other symmetries, potentials with  local maxima (so that  the smooth system has singular level sets of the Liouville foliation), and the influence of small perturbations and soft potentials are exciting directions to be further explored (see related results in  \cite{pnueli2018near,RK2014smooth}).

\section*{Acknowledgments}
 Supported by ISF 1208/16  and partially by the National Science Foundation under Grant No. 1440140, while one of the authors (VRK) was in residence at the Mathematical Sciences Research Institute in Berkeley, California, during the Fall semester of 2018.
LB, SE, BF and SGC visit and work was supported by the Bessie Laurence International Summer Science Institute at the Weizmann institute.
VRK thanks  K. Fr{\k{a}}czek  and B. Weiss for their valuable explanations on the dynamics on families of flat surfaces.

\section*{List of abbreviations}

\paragraph*{d.o.f} degrees-of-freedom

\paragraph*{HIS} Hamiltonian impact systems

\paragraph*{LIHIS} Liouville-integrable HIS
\paragraph*{QIHIS} Quasi-integrable HIS

\paragraph*{IEM} Interval Exchange Map

\end{document}